\documentclass[a4paper]{amsart}
\usepackage{amsmath,amssymb,amsfonts,graphicx}
\usepackage[mathscr]{euscript}
\usepackage[all]{xy}
\usepackage{enumitem}
    \theoremstyle{plain}
\newtheorem{theorem}{Theorem}[section]
\newtheorem{lemma}[theorem]{Lemma}
\newtheorem{proposition}[theorem]{Proposition}
\newtheorem{corollary}[theorem]{Corollary}
\newtheorem{definition}[theorem]{Definition}

\newtheorem{Alocal}[theorem]{The $\A$-Local Structure Theorem}
\newtheorem{AVeech}[theorem]{$\A$-Veech's Theorem}
\theoremstyle{definition}
\newtheorem{example}[theorem]{Example}
\newtheorem{remark}[theorem]{Remark}
\newtheorem{remarks}[theorem]{Remarks}

\def\<#1>{\langle\, #1\,\rangle}

\newcommand{\squ}{\raisebox{.20mm}{\scalebox{0.5}{$\square$\,}}}
\newcommand\restr[2]{{
  \left.\kern-\nulldelimiterspace 
  #1 
  \vphantom{|} 
  \right|_{_{#2}} 
  }}

\newcommand{\I}{\mathscr{I}}
\newcommand{\R}{\mathbb{R}}
\newcommand{\C}{\mathbb{C}}
\newcommand{\T}{\mathbb{T}}
\newcommand{\N}{\mathbb{N}}

\newcommand{\B}{\mathscr{B}}
\newcommand{\A}{\mathscr{A}}

\newcommand{\Z}{\mathbb{Z}}

\DeclareMathOperator{\supp}{supp}

\newcommand{\cnaught}{{\mathit{C}_0(G)}}

\newcommand{\luc}{\mathit{LUC}(G)}

\renewcommand{\emptyset}{\varnothing}

\font\seis=cmr6
\def\CB{\mathscr{CB}}
\def\ruc{{\seis{\mathscr{RUC}}}}
\def\uc{\mathscr{UC}}
\def\luc{{\seis{\mathscr{LUC}}}}
\def\wap{{\seis{\mathscr{WAP}}}}
\def\lc{{\seis{\mathscr{LC}}}}
\def\B{{\mathscr{B}}}

\def\ap{{\seis{\mathscr{AP}}}}
\def\lwap{{\seis{\luc_\ast}}}

\begin{document}


\title[Pym's and Veech's Theorems]{Algebraic structure of semigroup compactifications: Pym's and Veech's Theorems and  strongly prime points
}

\author[Filali and Galindo]{M. Filali \and  J. Galindo}

\thanks{Research of the second  author  partially supported by
 Spanish MICINN $\mbox{ grant MTM2016-77143-P}$}
\address{\noindent Mahmoud Filali,
Department of Mathematical Sciences\\University of Oulu\\Oulu,
Finland. \hfill\break \noindent E-mail: {\tt mfilali@cc.oulu.fi}}
\address{\noindent Jorge Galindo, Instituto Universitario de Matem\'aticas y
Aplicaciones (IMAC)\\\\ Universidad Jaume I, E-12071, Cas\-tell\'on,
Spain. \hfill\break \noindent E-mail: {\tt jgalindo@mat.uji.es}}


\date{\today}


\begin{abstract}
The spectrum of an admissible subalgebra $\A(G)$ of  $\luc(G)$, the algebra of  right uniformly continuous functions on a locally compact group $G$, constitutes a semigroup compactification $G^\A$ of $G$. In this paper we analyze the algebraic behaviour of those  points  of $G^\A$ that  lie in the closure of $\A(G)$-sets, sets
whose  characteristic function can be approximated  by functions in $\A(G)$.

This analysis  provides a common ground for far reaching generalizations of  Veech's property (the action of $G$ on $G^\luc$ is free)  and Pym's Local Structure Theorem.  This approach is  linked to  the concept of translation-compact set,  recently developed by the authors, and leads to characterizations of stronlgy prime points in $G^\A$,
 points that do not belong to the closure of $G^*G^*$, where $G^\ast=G^\A\setminus G$. All these results will be applied to show that, in many of  the most important algebras,  left invariant means  of $\A(G)$ (when such means are present) are supported in the closure of $G^*G^*$.
\end{abstract}

\maketitle

\section{Introduction}
Let $\A(G)$ be a $C^\ast$-algebra of complex-valued
functions on a locally compact group $G$.
  Under mild assumptions (usually summarized with  the term \emph{admissible}, see Section \ref{prelim}) the spectrum of $\A(G)$, here denoted as $G^\A$ to emphasize its role as a compactification of $G$,  can  be made into a right topological semigroup. The  Gelfand map becomes then a semigroup isomorphism of $G$ onto a dense subgroup of $G^\A$.

  The algebraic structure of  the semigroup $G^\A$ is rather involved, as algebraic properties of $G^\A$ are  usually deeply woven into the properties of $\A(G)$ and  the combinatorics of $G.$
In this paper we address some of these properties. Our starting point are two  of  maybe the most important theorems on the  algebraic structure of $G^\luc$:  Veech's property and Pym's Local Structure Theorem.

\begin{theorem}[Veech's Theorem, \cite{V}]\label{veech}
Let $G$ be a locally compact group. Then $sx\neq x$ for every $s\in G$, $s\neq e$ and every $x\in G^\luc$.
\end{theorem}

\begin{theorem}[Pym's Local Structure Theorem \cite{Pym}]\label{LST}
Let U be an open symmetric neighbourdood of $e$ with compact closure in $G$. Let $X \subseteq G $ be maximal with respect to the
property that $\{Ux\colon x\in X\}$ is a disjoint family. Then $\overline{X}$ (the closure in $G^\luc$) is
homeomorphic with the Stone-\v Cech compactification $\beta X$ of $X$     and for each open neighbourhood $V$ of $e$ in $ G$ with $\overline V \subseteq  U$,
the subspace $V\overline{X}$ is open in $G^\luc$ and homeomorphic with $V \times \beta X$.  Moreover, given
any $x\in G^\luc$,  $U $ and $X $ can be chosen so that $x\in    \overline{X}$.
\end{theorem}

Veech's theorem was first proved in the context of topological  dynamics, see the references in \cite[page 393]{vries}, and  turned out to be
an  essential tool for the theory of semigroup compactifications, see for instance  \cite{F1}, \cite{F07}, \cite{HS}, \cite{LMP}. For  discrete $G$, it can be obtained  as  a direct consequence    of a set theoretic partition lemma called the \emph{Three Sets Lemma}, already used by  Ellis \cite{E} (see \cite{Pym} for a direct  proof of the   version of this lemma adapted to the present context).

The first application of Pym's Local Structure Theorem (and probably  its  original motivation) was towards  a simplified proof of Veech's Theorem consisting of a  combination of  the Local Structure Theorem with the Three Sets Lemma. Its significance however goes far beyond this  application, as  it bridges  the discrete and    locally compact cases  in  quite a precise way.

Our first objective in the sections that follow is to  obtain general versions of both Veech's Theorem and Pym's  Local Structure Theorem. We find that beneath both theorems lies a common structure, that of  $\A(G)$-sets. When $G$ is discrete, a subset $T\subset G$ will be  said to be  an $\A(G)$-set if the  characteristic functions $1_{T_1}\in \A(G)$, for every $T_1\subseteq T$.  For nondiscrete $G$ the concept has to be suitably adapted, see  Section \ref{prelim}. The recourse to $\A(G)$-sets allows us to deal with the analogs of Theorems \ref{veech} and \ref{LST} in considerable generality, encompassing  not only $\luc(G)$ but any unital left invariant $C^\ast$-subalgebra of $\luc(G)$, and sometimes even general topological groups.   Veech's Theorem and Pym's  Local Structure Theorem will be addressed in sections \ref{sec:veech} and \ref{sec:pym},  after  our knowledge of $\A(G)$-sets is developed in sections \ref{sec:ag} and \ref{sec:lwap}. Other theorems related to the Local Structure Theorems proved in \cite{BP}, such as the Two-Point and the Compact-set Local Structure Theorems, follow as well from  our description of the structure of $G^\A$ around closures of $\A(G)$-sets.

 A difference between our approach and that of Pym in \cite{Pym} is that, instead of deducing  Veech's property from  the Local Structure Theorem, we  see both theorems as complementary descriptions of the behaviour, local in one case  and global in the other, of   points of $G^\A$  that are  in the closure of some $\A(G)$-set. The properties of $\A(G)$-sets (together with the Three Sets Lemma in the case of Veech's property) yield independent, elementary proofs of both theorems.

Our work on  the right-topological semigroup structure at   those  points of $G^\A$ that lie in the closure of $\A(G)$-sets is pushed further at the end of Section \ref{sec:pym} to  find subsets of $G^\A$ larger than $G$ whose points
are injectivity points (i.e., points whose action by multiplication on $G^\A$ is free).

$\A(G)$-sets are applied further in  Section \ref{sec:sp} to characterize  strongly prime points, i.e., points  of $G^\A$ that are not in the closure $\overline{G^{\ast}G^\ast}$, where $G^\ast=G^\A\setminus G$. A key tool here is the algebra  $\luc_\ast(G)$, an admissible algebra  introduced in this paper whose interpolation sets have  interesting combinatorial properties.

 In our last section, we deduce from  the results obtained in sections 4, 6 and 7, that the support of any left invariant mean on $\A(G)$ is  contained in $\overline{G^{\ast}G^\ast}$ for a large family of  amenable algebras $\A(G)$, including $\luc(G)$ for any amenable locally compact group $G$,  all algebras  containing $\wap(G)$ for Abelian $G$ and $\luc_\ast(G)$ for any locally compact group $G$.  The case  generalizes a result obtained by  Dales, Lau  and Strauss \cite{DLS} for $G$ discrete and $\A(G)=\ell_\infty(G)$.

\section{The algebras $\A(G)$ and the compactifications $G^\A$}
\label{prelim}
We recall here  the necessary definitions concerning semigroup compactifications that are used throughout  the paper. As long as possible, we follow  notation and terminology from  \cite{BJM} and \cite{FG1}.
Throughout the paper $G$ is a Hausdorff topological group with identity $e$, most of the time  locally compact.

We use $\CB(G)$ to denote the $C^\ast$-algebra of continuous bounded scalar-valued functions on $G$
equipped with the supremum norm, and $C_0(G)$ to denote  the $C^\ast$-subalgebra of the functions in $\CB(G)$ vanishing at infinity.

For each function $f$ defined on $G$ and $s\in G$, let ${}_sf$ be the left translate of
$f$ by $s$, this is defined on $G$ by ${}_sf(t)=f(st)$.

Following \cite{hewiross}, a  bounded function $f$ on $G$
is  {\it right uniformly continuous} when, for every $\epsilon>0$,
there exists a neighbourhood $U$ of $e$ such that
\[
|f(s)-f(t)|<\epsilon\quad\text{whenever}\quad st^{-1}\in U.\]
These are functions which are left norm continuous in the sense that
 \[s\mapsto {}_sf: G\to \CB(G)\] is continuous.  In the literature, the $C^*$-algebra formed by these functions is denoted  by $C_{ru}(G)$,  $LC(G)$, $\lc(G)$ or  $\luc(G)$. In this paper, we shall use the latter notation.

The $C^*$-algebra $\ruc(G)$ is defined in the same way replacing  the right  uniformity of $G$ by the left uniformity. Again, the notation
$\ruc(G)$ stresses the fact that $f\in \ruc(G)$ if and only if the map  \[s\mapsto f_s: G\to \CB(G)\] is continuous, where now $f_s$
is the right translate of $f$ by $s$. We let $\uc(G)=\luc(G)\cap \ruc(G).$

A function $f$ is {\it weakly almost periodic} when the set of all
its left (equivalently, right) translates forms a relatively weakly
compact subset in $\CB(G).$
When the set of all
the left (equivalently, right) translates of  a function $f$  forms a relatively norm-compact subset of $\CB(G)$, the  function $f$ is said to be  {\it almost periodic}. The almost periodic and weakly almost periodic functions   constitute $C^*$-subalgebras of $\CB(G)$ which we denote by $\ap(G)$ and $\wap(G)$, respectively.

Let $\A(G)$ be a $C^\ast$-subalgebra of $\CB(G)$ and $G^\A$ be the spectrum (non-zero multiplicative linear functionals) of $\A(G)$ with the Gelfand topology.
There is a canonical morphism  $\epsilon_\A \colon G\to G^\A$  given by evaluations:
\[\epsilon_\A(s)(f)=f(s),\; \text{for every}\;f\in \A(G)\; \text{and}\; s\in G.\] We shall use $G^\ast$ to denote \emph{the remainder} $G^\A\setminus \epsilon_\A(G).$ The map $\epsilon_\A$
is continuous if  $\A(G)\subseteq \CB(G)$, and it is injective if and only if $\A(G)$ separates
the points of $G$.
It is a routine to check that the map $\epsilon_\A: G\to G^\A$ is a homeomorphism when  $G$ is locally compact and  $C_0(G)\subseteq \A(G).$

As known, Gelfand duality identifies $\A(G)$ with $\CB(G^\A)$.
So to every $f\in \A(G)$ there corresponds a function $f^\A\in \CB(G^\A)$ such that $f^\A\circ \epsilon_\A=f.$
When $\epsilon_\A$ is injective, $f^\A$ can be seen as a \emph{continuous extension} of $f$ to $G^\A$.

We say that a subalgebra $\A(G)$ of $\CB(G)$ is \emph{left translation invariant} when ${}_sf\in\A(G)$ for every $f\in \A(G)$ and $s\in G.$
In  such algebras, there is a natural action of $G$ on $G^\A$
given by \[G\times G^\A \to G^\A: (s,x)\mapsto \epsilon_\A(s)x,\] where $\epsilon_\A(s)x(f)=x({}_sf)$  for every $s\in G$, $x\in G^\A$ and $f\in\A(G).$
Right translation invariance is defined analogously. We say that a subalgebra $\A(G)$ of $\CB(G)$ is  translation invariant when it is both left and right translation invariant.

We say that the $C^\ast$-subalgebra $\A(G)$ of $\CB(G)$ is {\it admissible} when it satisfies the following properties:
(i) $1\in \A(G)$,
(ii) $\A(G)$ is translation invariant,
and (iii) for each $f\in\A(G)$ and every $x\in G^\A$, the function ${}_xf$ defined on $G$ by the equality ${}_xf(s)=\epsilon_\A(s)x(f)=x({}_sf)$ is in $\A(G)$.
 When $\A(G)$ is admissible,
  it is possible to introduce a binary operation on $G^\A$  by the rule
 \[xy(f)=x({}_yf)\quad \text{for every}\quad x,y\in G^\A\;\text{and}\; f\in\A(G).\]
With this operation $G^\A$ is a compact semigroup and  $\epsilon_\A$ is   a semigroup homomorphism with $\epsilon_\A(G)$  dense in $G^\A$. Furthermore,
 the mappings
\[
x\mapsto xy\colon  G^\A\rightarrow  G^\A
\,\,\text{ and } \,\,
x\mapsto \epsilon_\A(s)x\colon G^\A\rightarrow  G^\A
\]
are continuous for every $y\in G^\A$ and $s\in G$.
In other words,   $G^\A$ becomes a semigroup compactification of the group $G$. We will refer to it as the {\it $\A$-compactification of $G$}.

All these concepts admit a right analog. If $\A(G)$ is a translation invariant, unital $C^*$-subalgebra of $\CB(G)$,
and   the function defined on $G$ by $f_x(s)=x\epsilon_\A(s)(f)=x(f_s)$ is in $\A(G)$ for every $x\in G^\A$ and $f\in\A(G)$ (where now $f_s$ is the right translate of $f$ by $s$), then  $G^\A$ becomes a compact semigroup
 with the product given by
\[ x\squ y(f)=y(f_x)\quad \text{for every}\quad x,y\in G^\A\;\text{and}\; f\in\A(G).\]
  With this product, the mappings
\[
y\mapsto x\squ y\colon  G^\A\rightarrow  G^\A
\,\,\text{ and } \,\,
y\mapsto y\squ \epsilon_\A(s)\colon G^\A\rightarrow  G^\A
\]
are continuous for every $x\in G^\A$ and $s\in G$.

With any of the two products, the $\wap$-compactification $G^\wap$ is the largest semitopological semigroup compactification, while the
Bohr or  $\ap$-compactifiaction is the largest topological (semi)group compactification.

The $\luc$-compactification $G^{\luc}$ is the
largest semigroup compactification in the sense that any other
semigroup compactification appears as a  quotient of $G^{\luc}.$

 When  $G$ is discrete,  $\luc(G)=\ruc(G)=\ell_\infty(G)$ and $G^{\luc}=G^\ruc$ coincides  with the Stone-\v Cech compactification $\beta G$ of $G.$
 In this case, the reader is directed to \cite{HS} for the algebra in $\beta G$ and its applications to Ramsey theory.

To avoid cumbersomeness,  we will in general omit the function $\epsilon_\A$.
It will appear only when there is more than one compactification under the spotlight as it is the case in Subsection \ref{injectivity}.
So for instance,   $\epsilon_\A(s)x$ will be written as $sx$ when $s\in G$ and $x\in G^\A$,
and $\epsilon_\A(V)\overline{\epsilon_\A(T)}$ will be written simply as $V\overline T$ when $V,$ $T\subseteq G$ and the closure is taken in $G^\A$.

\subsection{$SIN$-groups, $E$-groups and $\wap(G)$}
Our results on $G^\wap$ usually require that $\wap(G)$ is rich enough. One way to guarantee this is to restrict $G$ to the classes of $SIN$- or $E$-groups, whose definitions we recall here.

A topological group $G$ is said to be an \emph{$SIN$-group}
(acronym for small invariant neighbourhoods) when the left and the right uniformities on $G$ coincide, i.e., if $G$ has a base at $e$ made of open sets invariant under all inner automorphisms.

A   non-relatively compact $X\subset G$  is said to be an \emph{$E$-set} if, for
each neighbourhood $U$ of $e,$ the set \[{\displaystyle \bigcap \{x^{-1}Ux: t\in
X\cup X^{-1}\}}\] is again a neighbourhood of $e.$
A locally compact group $G$
is an {\it E-group} if it contains an  $E$-set $X$.
This is a large class of locally compact groups which  contains properly the noncompact $SIN$-groups,
see \cite{C1}.

\subsection{Invariant means and $\A_0(G)$ algebras.}\label{sub0}
    {\it A left invariant mean} on a unital left translation invariant subspace $\A(G)$ of $\CB(G)$ is a positive functional $\mu\in \A(G)^*$ such that $\mu(1) = 1,$ and
$\mu({}_sf) = \mu(f)$ for  all $f\in\A(G)$ and $s \in G.$ Right invariant means are defined analogously.
A mean is invariant when it is both left and right invariant.
When a left invariant mean exists in $\A(G)^*,$ we say that $\A(G)$ is amenable.
While $\ap(G)$ and $\wap(G)$ are always amenable (with the invariant mean being the Haar measure on the $\ap$-compactification $G^\ap$ of
$G$), $\luc(G)$ may not be amenable as it is the case when $G$ is the free group $\mathbb F_2$. Invariant means  on semigroups have been extensively studied since Day's seminal paper \cite{Day}.
We direct the reader to \cite{Paterson} and \cite{Pier} for more details.

 To every   amenable $C^\ast$-subalgebra $\A(G)$  of $\CB(G)$ one can associate the algebra  $\A_0(G)$ consisting of those $f\in \A(G)$ such  that $\mu(|f|)=0)$ for every  left invariant mean $\mu$ in $\A(G)^*$. Note that $\A_0(G)$ is always an ideal in $\A(G)$.

We will at several points refer to the support of invariant means.  Since, by the Riesz representation theorem,  means can  be regarded as  Borel measures on $G^\A$, we have two equivalent ways to define the   support  of a positive  element $\mu$ of $\A(G)^*$:
\begin{align*}\supp(\mu) &= \{x \in G^\A : \mu(f)\ne 0 \;\text{whenever}\;  f\in \A(G), f \ge 0\;\text{ and}\; f^\A(x) \ne 0\}\\&=
G^\A \setminus \{O : O\;\text{ open in}\;  G^\A \;\text{and}\; \mu(O)=0\},
\end{align*} depending on whether $\mu$ is seen as a functional or as a measure.

\section{$\A(G)$-sets}\label{sec:ag}
We introduce in this section a family of sets that will be important  for the remainder of the paper,  $\A(G)$-sets. These sets are closely related to approximable interpolation sets that were introduced in \cite{FG1}, see this reference and \cite{FG2} for more details on this class. We review here some  basic properties of these sets that will be needed later. We in particular observe that  a right uniformly discrete set (the case of interest for our applications) is an  $\A(G)$-set if and only if it is an  approximable $\A(G)$-interpolation set. Thus the concept  of $\A(G)$-set can be seen as a simplification of the concept of $\A(G)$-interpolation set.

\begin{definition} \label{approx} Let $G$ be a topological group and $\A(G)\subseteq \CB(G)$.
A subset $T\subseteq G$ is said to be
\begin{enumerate}
  \item an \emph{$\A(G)$-interpolation set} if every bounded function
$f\colon T\to \C$ can be extended to a function $\overline{f}\colon G\to
\C$ such that $\overline{f}\in \A(G)$.
\item an $\A(G)$-set when  for every open
 neighbourhood $U$ of $e$, there is an open neighbourhood $V$ of $e$ with $\overline{V}\subseteq U$ such
that, for each $T_1\subseteq T$  there is   $h\in \A(G)$ with
$h(VT_1)=\{1\}$ and $h(G\setminus (UT_1))=\{0\}$.
\item an \emph{approximable $\A(G)$-interpolation set} if it is both an
 $\A(G)$-interpolation set and an $\A(G)$-set.
\end{enumerate}
\end{definition}

In addition to these general definitions,  we need also to recall the known notion of uniform discreteness.
Along with translation-finite and translation-compact sets, these sets determine combinatorially the $\luc(G)$- and the $\wap(G)$-sets.
 The latter sets are essential for the  next and the last two sections.

 \begin{definition} \label{def:rud}
   Let $G$ be a topological group  and $U$ be a neighbourhood of the identity $e$.
We say that a subset $T$ of $G$ is \emph{right $U$-uniformly
discrete}  if \[Ut\cap Ut^\prime=\emptyset\quad\text{ for every}\quad t\neq
t^\prime\in T.\] \emph{Left $U$-uniformly discrete} sets are defined analogously.

We say that $T$ is \emph{right (left) uniformly discrete} when it is right (left) $U$-uniformly
discrete for some  neighbourhood $U$ of $e$.
\end{definition}

\begin{lemma}[Lemma 4.8(i) and Theorem 4.9 of \cite{FG1}]\label{udluc}
Let $G$ be a topological group and let $T\subseteq G$:
\begin{enumerate}[label=\textup{(\roman{enumi})}]
\item \label{i}   If $T$ is  right (respectively,  left) uniformly discrete, then $T$ is an approximable $\luc(G)$ (resp. $\ruc(G)$)-interpolation set.
\item If $G$ is metrizable, then $T$  is  an approximable $\luc(G)$(resp. $\ruc(G)$)-interpolation  set  if and only if it is right (resp.  left) uniformly discrete.
        \end{enumerate}
        \end{lemma}

\begin{proposition}\label{apg=ag} Let $G$ be a topological  group, $\A(G)$ be a unital $C ^\ast$-subalgebra of
$\CB(G)$
and let $T$ be a right uniformly discrete subset of $G.$
Then $T$ is an $\A(G)$-set if and only if it is an approximable $\A(G)$-interpolation set.
\end{proposition}

\begin{proof} Let $U$ be a neighbourhood of $e$ and suppose that $T$ is right $U$-uniformly discrete.
We check first that disjoint subsets of $T$ have disjoint closures in $G^\A.$
 Let $T_1$ and $T_2$ be disjoint subsets of $T$.
Pick a neighbourhood  $V$  of $e$ with $\overline{V}\subseteq U $ and  a function $h\in\A(G)$ such that $h(VT_1)=\{1\}$
and $h(G\setminus U T_1)=\{0\}.$
Since $T$ is right $U$-uniformly discrete, we see that $U T_1\cap U T_2=\emptyset$, i.e.,
$UT_2\subseteq G\setminus UT_1.$ We have then that  $h(T_1)=1$ while $h(T_2)=0$. Thus, $\overline{T_1}\cap\overline{T_2}=\emptyset$ (the closure is taken of course in $G^\A$).

As known and easy to check (see for example \cite[Corollary 3.2.1]{EN}), this proves that  bounded functions on $T$ extend to continuous functions on $\overline T,$ and so to the whole $G^\A.$
In other words, every bounded function on $T$ will extend to the whole group $G$ as a function in $\A(G),$ showing that $T$ is an  approximable
$\A$-interpolation set. The converse is clear.
\end{proof}

\begin{corollary} \label{cor:rudA} Let $G$ be a metrizable topological group and $\A(G)$ be a unital $C ^\ast$-subalgebra of
$\luc(G)$. Then $T$ is an approximable $\A(G)$-interpolation set if and only if it is a right uniformly discrete  $\A(G)$-set.
\end{corollary}

\begin{proof} Since $\A(G)\subseteq \luc(G),$ we know from Lemma \ref{udluc} that approximable $\A(G)$-interpolation sets are necessarily right uniformly discrete.
\end{proof}

\begin{remark}
 \begin{enumerate}
    \item
 In the non-metrizable case, Corollary \ref{cor:rudA} is not true as already observed in \cite[Example 6.2]{FG1}.

\item In general, $\A(G)$-sets (even discrete $\A(G)$-sets) are not necessarily approximable $\A(G)$-interpolation sets. Consider a  nondiscrete locally compact group $G$ and  an admissible subalgebra  $\A(G)$ of $\luc(G)$ such that $C_0(G)\subseteq \A(G)$. Let $T$ be a nondiscrete infinite relatively compact subset of $G$  and let $U$ be any open   neighbourhood of $e$. Pick a relatively compact symmetric neighbourhood $W$ of $e$ such that $W^2\subseteq U$ and let $V$ be a relatively compact neighbourhood of $e$
such that $\overline V\subseteq W$. Let $T_1$ be any subset of $T.$ Since $\overline{T_1}\subset WT_1$, then
 \[\overline{VT_1}\subseteq \overline{V}WT\subseteq W^2T_1\subseteq UT_1,\]
where all closures are taken in $G$. We may take  then $h\in C_0(G)$ such that  $h(VT_1)=\{1\}$ and $h(G\setminus UT_1)=\{0\}.$ It follows that  $T$ is an $\A(G)$-set. Not being discrete, it cannot be an $\A(G)$-interpolation set. When $G$ is metrizable, we do not even need to require that $T$ is nondiscrete, any infinite relatively compact subset of $G$ will do by  Corollary \ref{cor:rudA}.
 \end{enumerate}
\end{remark}

\begin{definition}\label{rud}
Let $G$ be a topological  group
and let $T$ be a subset of $G$.
\begin{enumerate}
\item
$T$ is \emph{ right translation-finite} if every  infinite subset $L$ of $G$
contains a finite subset $F$ such that $\bigcap\{b^{-1}T\colon b\in
F\}$ is finite.
\item
$T$ is \emph{ right translation-compact} if every  non-relatively compact subset $L$ of $G$
contains a finite subset $F$ such that $\bigcap\{b^{-1}T\colon b\in
F\}$ is relatively compact.
\item $T$ is a \emph{right t-set}
if there exists a compact subset $K$ of $G$ containing the identity $e$ such that $gT\cap T$
is
relatively compact for every $g\notin K$.
\item Left translation-finite sets, left translation-compact sets, and left t-sets are defined analogously.
 \item
 $T$ is
 \emph{ translation-finite} when it is both right and left translation-finite, {\it translation-compact} when it is both right and left translation-compact, and it is a \emph{t-set} when it is both a  right and a left t-set.
\end{enumerate}

\end{definition}
\begin{remarks}$\;$ \smallskip
  \begin{enumerate}
    \item The term translation-finite was first introduced by Ruppert \cite{rup}. Chou  used the term $R_W$-sets in \cite{C4}. Right translation-finite sets were termed
\emph{left sparse} in  \cite{FP}. (Right) translation-compact sets were introduced in \cite{FG1}.
\item
It should be remarked that, while finite unions of right (or left) t-sets are clearly right (left) translation-compact sets,  the converse is not true even if $G$ is discrete, see \cite[Section 3]{C4} and \cite[ Examples 11]{rup}.
\item
Right (Left) translation-finite sets are clearly right (left) translation-compact in any topological  group.
It is easy to see that the converse does not hold in general.
But the converse is true if the sets are right (left) uniformly discrete sets as next proposition shows.
  \end{enumerate}
\end{remarks}

\proposition\label{prop}
Let $G$ be a topological  group, $U$ be a symmetric neighbourhood of the identity and $T$ be a right (left)
$U$-uniformly discrete subset of $G$. Then $T$ is right (left) translation-compact if and only if
$T$ is right (left) translation-finite.
\endproposition

\proof Suppose that $T$ is not  right translation-finite, and let $L$
be an infinite subset of $G$ such that $\bigcap_{b\in F}b^{-1}T$ is infinite for every finite subset $F$ of $L$. Since $T$ is discrete, the sets  $\bigcap_{b\in F}b^{-1}T$ are not relatively compact, otherwise they would be finite.
So it is enough to check that $L$ is not relatively compact.
Suppose, otherwise, that $L$ is relatively compact, and let $\{Ux_1,Ux_2,...,Ux_n\}$ be a finite cover of $L$ by finitely many translates of $U$ in $G$. We claim, that for each $1\le i\le n$, the set  $L\cap Ux_i$ contains at most one point. For, if $u,$ $v\in U$ are such that $ux_i,$ $vx_i\in L$, then
$x_i^{-1}u^{-1}T\cap x_i^{-1}v^{-1}T\ne\emptyset$ (remember that $\cap_{b\in F}b^{-1}T$ is infinite for every finite subset $F$ of $L$) implies that $u^{-1}T\cap v^{-1}T\ne\emptyset$, and so $u^{-1}t_1=v^{-1}t_2$ for some $t_1,$ $t_2\in T.$ Hence, $t_1=t_2$ since $T$ is right $U$-uniformly discrete, and so $u=v$, as required. But this is impossible since  $L$ is infinite.
\endproof

We next indicate how to construct  a non-relatively compact t-set, and so  a translation-compact set, in  a noncompact locally compact group.

\begin{example}\label{t-set}
Let $G$  be a noncompact locally compact group and  let $X\subseteq G$ be non-relatively compact. Fix a compact neighbourhood $V$ of the identity $e$. Start with $x_0=e,$ say. Suppose that the elements $x_n$ have been picked for every
$n<n_0$. Set ${\displaystyle T_{n_0}= \bigcup\limits_{n_1,n_2,n_3<n_0}V^2x_{n_1}^{\epsilon_1}x_{n_2}^{\epsilon_2}
V^2x_{n_3}^{\epsilon_3},}$
where each $\epsilon_i=\pm1.$
Since $T_{n_0}$ is compact
we may pick another element $x_{n_0}\in X\setminus T_{n_0}$ for our set $T$. In such a way, we form a set $T=\{x_n:n<\omega\}$.
As already checked in \cite{F07}, the set $T$ is right $V^2$-uniformly discrete and both $s(VT) \cap (VT)$ and $(VT)s \cap (VT)$ are relatively compact for every $s \not\in V^2$, i.e., $VT$ is a t-set.

The construction can be easily adapted for the set $T$ to have the  cardinality of the smallest covering of $X$ by compact subsets.\end{example}

\section{Right translation-compact sets and the related algebra}\label{sec:lwap}
Before we get to the core of the paper, we introduce in this section
a new admissible subalgebra of $\luc(G)$, which we denote by $\luc_\ast(G).$
We show that the right translation-compact sets are precisely the approximable $\lwap(G)$-interpolation sets.
Lemma \ref{general**},  based on Theorem 7 of \cite{rup},  is the key to obtain this characterization. This lemma will also be very useful in Sections \ref{sec:sp} and \ref{0meansets}.
Our second application of Lemma \ref{general**} in this section is Theorem \ref{wap=wap_0}. This is one of the main results obtained in our previous article, namely \cite[Theorem 4.22]{FG1}. The proof there  was involved and required some pretty heavy combinatorial tools. Lemma \ref{general**} furnishes a more direct approach from which  an even  more general result can be  obtained.

We start with a collection of lemmas needed for this section and for the rest of the paper. The first  two of them characterize
 $\wap(G)$-interpolation sets.
 A detailed study  of these sets  is carried  out in our recent paper  \cite{FG1}.

 \begin{lemma}\label{lem:lucint}
Let $G$ be an $E$-group and let $U$ be a neighbourhood of the identity. If $T\subseteq G$ is a right (or left) $U$-uniformly discrete $E$-set  such that $UT$ is translation-compact, then $T$ is an approximable $\wap(G)$-interpolation set.
\end{lemma}

\begin{proof} The proof is  a slight modification of that given in \cite[Lemma 4.8(ii)]{FG1}. By Proposition \ref{apg=ag}, it is enough to check that $T$ is a $\wap(G)$-set.
Let $U_0$ be an arbitrary neighbourhood of $e$,  $V$ and $W$ be symmetric neighbourhoods of $e$ such that
$\overline V\subseteq W\subseteq W^2\subseteq U_0\cap U$,  and $T_1\subseteq T.$ Let $\psi\in \luc(G)$ with
$\psi(\overline V)=1$ and $\psi(G\setminus W)=\{0\}$.
 Then as seen in \cite[Lemma 4.6]{FG1}, the function $h=1_{T_1,\psi}=\sum_{t\in
T_1}\psi_{t^{-1}}\in\uc(G)$. Since $h(G\setminus WT_1)=\{0\}$ and $WT_1$ is translation-compact, the function
$h\in\wap(G)$ by \cite[Lemma 4.3]{FG1}. Since clearly $h(G\setminus U_0T_1)=\{0\}$ and $h(VT_1)=\{1\}$,
the claim follows.
\end{proof}

Next Lemma is a slight variant of \cite[Corollary 4.12]{FG1}.

 \begin{lemma}
   \label{cor:wap2}
   Let $G$ be a locally compact group and $\A(G)$ be an admissible subalgebra of $\wap(G)$.
     Let $T$ be a
right  $U$-uniformly discrete subset of $G$ for some symmetric relatively compact
neighbourhood $U$ of $e$. If $T$ is an $\A(G)$-set, then $UT$ is translation-compact.
 \end{lemma}

\begin{proof} Suppose that $UT$ is not right translation-compact.
We check that $T$ is not a $\wap(G)$-set. Let $V$ be any open neighbourhood of $e$ with $\overline{V}\subseteq U$.
Since $UT$ is not right translation-compact, we may apply
 \cite[Lemma 4.11]{FG1} to find $T_1\subseteq T$
  such that no bounded function  $f\colon G\to \C$ with
  $f(T_1)=\{1\}$
  and $f(G\setminus UT_1)=\{0\}$ is  weakly almost periodic. Therefore,
   $T$ is not a $\wap(G)$-set, and so it cannot be an $\A(G)$-set either.

The argument is symmetric if we suppose that $UT$ is not left translation-compact.
\end{proof}

\begin{remark}\label{typo}
In the statement of  several Theorems of \cite{FG1}, translation-compactness was  erroneously exchanged with
 right translation-compactness. This affects to  theorems
4.15 and 4.16 of [loc. cit.]  where  the set  $VT$ appearing there should have been required to be
translation-compact instead of just right translation-compact.
The properties actually used in the  proofs of these theorems are the following: Theorem 4.16 relies on Theorem 4.15 and the latter relies on Lemma 4.3 that requires the set to be translation-compact. The proof of Lemma 4.11, on the other hand, works for one sided (either right or left) translation-compact sets.\end{remark}

Recall that $\epsilon_\A : G\to G^\A$ is the natural
mapping, and  $G^\ast=G^\A\setminus\epsilon_\A(G)$.
The following lemma will be relevant when $\A(G)\subseteq \wap(G)$.

\begin{lemma}[Corollary  2.9 of \cite{deleglick} or Theorem 4.2.14 of \cite{BJM}] \label{im}
Let $\A(G)$  be an admissible subalgebra of $\wap(G)$. The support of the (unique) invariant mean of $\A(G)$ is then the minimal ideal of $G^\A$, which  is a compact subgroup of $G^\ast=G^\A \setminus\epsilon_A( G)$. It is thus contained in $G^\ast G^\ast$.
\end{lemma}

    Remainders of semigroup compactifications are often ideals.
We summarize the simplest cases in the next two lemmas. The first one is probably well-known. We omit the easy proof.

\begin{lemma} \label{ideal} Let $G$ be a locally compact group, $\A(G)$ be an admissible subalgebra of $\luc(G)$.
Then $\epsilon_\A(G)$ is open in $G^\A$ if and only if
$G^\ast$ is a closed two-sided ideal in $G^\A.$
\end{lemma}



\begin{lemma} \label{homeo} Let $G$ be a locally compact group, $\A(G)$ be an admissible
subalgebra of $\luc(G)$. Then the following statements are equivalent.
\begin{enumerate}
\item  $\epsilon_\A$ is a homeomorphism of $G$ onto $\epsilon_\A(G)$.
\item $\epsilon_\A$ is injective and $\epsilon_\A(G)$ is open in $G^\A.$
  \item  $C_0(G)\subseteq \A(G)$.
\item $\epsilon_\A$ is injective and $G^\ast$ is a closed two-sided ideal of $G^\A.$
\end{enumerate}
\end{lemma}

\begin{proof} The equivalence of the three first statements follows as in the case of $\A(G)=\wap(G)$, see
 \cite[Proposition III.4.5]{BH} or \cite[Theorem 3.6]{B}.

The rest follows from Lemma \ref{ideal}.

\end{proof}

\begin{remark} 1. When $C_0(G)$ is not contained in $\A(G),$ the map $\epsilon_\A$ may be injective but it is not a homeomorphism as it is the case when $\A(G)=\ap(G)$ and $G$ is a locally compact Abelian group.

2. When $G$ is not locally compact,
 $\epsilon_\A$ may fail to be  a homeomorphism also when $\A(G)=\wap(G)$. In fact,  $G^\A$ might even be a singleton.  For example, if $G$ is  the group of all orientation-preserving self-homeomorphisms of
$ [0,1]$, endowed with the compact-open topology, then $\wap(G)=\C1$, see \cite{Me}.

3. The map $\epsilon_\luc$ is a homeomorphism onto $\epsilon_\luc(G)$ for every topological group, see \cite{T} or \cite{gali10}, but $\epsilon_\luc(G)$ is  open in $G^\luc $ if and only if $G$ is locally compact.\end{remark}

\begin{lemma} \label{general} Let $G$ be a locally compact group, $\A(G)$ be an admissible subalgebra of $\CB(G)$,  $I$ be a proper closed two-sided ideal of $G^\A$, and put \[I^\perp=\left\{f\in \A(G):f^\A(I)=\{0\}\right\}.\]
Then $\I(G)=I^\perp\oplus\C1$ is an admissible  subalgebra of $\A(G)$  containing  $C_0(G)\cap\A(G)$.
\end{lemma}

\begin{proof}  Since, necessarily $I\subseteq G^\ast$,
it is clear that $I^\perp$ is translation invariant $C^*$-subalgebra of $\A(G)$
 containing $\cnaught\cap\A(G).$
It remains to check that  if $f\in I^\perp,$ $a\in G^{\I}$, then the function defined on $G$ by ${}_af(s)=a({}_sf)$ is in
$I^\perp.$ It is a straightforward check that ${}_af\in \A(G).$
For this, note that if  $(s_\alpha)$ a net in $G$ with $(\epsilon_\I(s_\alpha))$ converging to $a$ in $G^{\I},$ then by taking a subnet we may assume that $(\epsilon_\A(s_\alpha))$ converges to $\bar a$ in $G^\A.$ Accordingly,
\[({}_af)(s)=a({}_sf)=\lim_\alpha f(ss_\alpha)= f^\A(s\bar a)=({}_{\bar a}f)(s)\;\text{ for every}\; s\in G,\] and so ${}_af={}_{\bar a}f\in \A(G).$

Let now $x\in I$ and $(x_\beta)$ be a net in $G$ converging to $x$ in $G^\A.$
Then
\[ ({}_af)^\A(x)=\lim_\beta \epsilon_\A(x_\beta) ({}_af)=\lim_\beta {\epsilon_\A(x_\beta)}a(f)=xa(f)=f^\A(x a)=0.\]
Therefore,  ${}_af\in I^\perp.$
\end{proof}

We now identify  the algebra for which right translation-compact sets are the approximable interpolation sets.
To avoid confusion on which remainder we are dealing with in our coming arguments, we put
 $G^{*l}=G^\luc\setminus G$,  $ G^{*r}=G^\ruc\setminus G$ and  $G^{*w}=G^\wap\setminus G$.
Let also $K(G^\wap)$ be the minimal ideal in $G^\wap.$
Recall that, by Lemma \ref{im}, $K(G^\wap)$ is the support of the invariant mean of $\wap(G)$.

\begin{definition} For a topological group $G,$
let \begin{align*}\luc_a(G)&=\overline{G^{*l}G^{*l}}^\perp=\left\{f\in \luc(G):f^\luc(G^{*l}G^{*l})=\{0\}\right\},\\
\ruc_a(G)&=\overline{G^{*r}\squ G^{*r}}^\perp=\left\{f\in \ruc(G):f^\ruc(G^{*r}\squ G^{*r})=\{0\}\right\},\\
\wap_a(G)&=\overline{G^{*w} G^{*w}}^\perp=\left\{f\in \wap(G):f^\wap(G^{*w} G^{*w})=\{0\}\right\},\\
\luc_\ast(G)&=\luc_a(G)\oplus\C1,\;\; \ruc_\ast(G)=\ruc_a(G)\oplus\C1
 \;\;\;\text{and}\;\;\;\wap_\ast(G)=\wap_a(G)\oplus\C1.\end{align*}
\end{definition}

\begin{lemma} \label{12Other} Let $G$ be a topological group, and let $\epsilon_l^w:G^\luc\to G^\wap$ and $\epsilon_r^w:G^\ruc\to G^\wap$ be
 the natural homomorphisms. Then
\begin{enumerate}
\item $f^\wap \circ \epsilon_l^w= f^\luc$ and $f^\wap \circ \epsilon_r^w=f^\ruc$
for every $f\in\wap(G)$.
\item If $G$ is locally compact, then $\epsilon_l^w(G^{*l})=\epsilon_r^w(G^{*r})=
G^{*w}.$
\end{enumerate}
\end{lemma}
\begin{proof}
The proof of the first statement follows from  the commutativity of the following diagrams:
 \begin{equation*}
\xymatrix{
G \ar[r]^{\epsilon_{_{\wap}}} \ar[d]^{\!\epsilon_{_{\luc}}} & G^\wap \\
 G^\luc \ar[ur]_{\epsilon^w_l} &}\qquad \hfill \qquad\xymatrix{
G \ar[r]^{\epsilon_{_{\wap}}} \ar[d]^{\!\epsilon_{_{\ruc}}} & G^\wap \\
 G^\ruc   \ar[ur]_{\epsilon^w_r} &}\end{equation*}

The second statement is now clear since $x$ is in the remainder of any of the three compactifications if and only if $x(f)=0$ for every
$f\in C_0(G).$
\end{proof}

\begin{proposition} \label{abelian} Let $G$ be a locally compact group. Then:
\begin{enumerate}
\item $\wap_\ast(G)$, $\luc_\ast(G)$
and  $\ruc_\ast(G)$  are admissible  subalgebras of $\wap(G),$ $\luc(G)$ and $\ruc(G)$, respectively,
and each contains  $C_0(G)$.
\item $\wap_\ast(G)\subseteq \luc_\ast(G)\cap \ruc_\ast(G)\subseteq \wap_0(G)\oplus\C1=K(G^\wap)^\perp\oplus\C1$.
\item If $G$ is Abelian, then $\luc_\ast(G)=\ruc_\ast(G).$
\end{enumerate}
\end{proposition}

\begin{proof} Statement (i) follows from Lemma \ref{general}.

As for the second statement, the first inclusion follows  easily
from  Lemma \ref{12Other}, and the equality follows from the definition.
 Moreover, if $f\in\luc_\ast(G)\cap \ruc_\ast(G)$, then $f\in \uc(G)=\luc(G)\cap \ruc(G).$
Suppose that $f$ is not constant, the claim is trivial otherwise.
We show first that
$f^\uc(xy)=f^\uc(x\squ y)$ for every $x,y\in G^\uc;$ that will imply that $f\in\wap(G).$
Let $(x_\alpha)$ and $(y_\beta)$ be two nets in $G$
converging, respectively, to $x$ and $y$ in $G^\uc.$ Let $x$ and $y$ denote also the cluster points of $(x_\alpha)$ and $(y_\beta)$ in both $G^\luc$ and $G^\ruc$. The points $x$ and $y$ are in $G^{*l}$ if and only if  the corresponding points are in $G^{*r}$, and so in this case,
\[f^\uc(xy)=f^\luc(xy)=0=f^\ruc(x\squ y)=f^\uc(x\squ y),\]
by assumption. If $x$, say, is in $G$ then \[f^\uc(xy)=\lim_\beta f^\uc(xy_\beta)=f^\uc(x\squ y).\]

 We now show that $f\in\wap_0(G)$. Since, by  Lemma \ref{im}, the invariant mean on $\wap(G)$ has its support contained in  $G^{*w}G^{*w},$  it will suffice show that $f^\wap\left(G^{*w}G^{*w}\right)=\{0\}$.
If $p,q \in G^{*w},$ we may pick $x$ and $y$ (see Lemma \ref{12Other})
in $G^{*l}$ with $\epsilon_l^w(x)=p$ and $\epsilon_l^w(y)=q$,  then \[f^\wap(pq)=f^\wap\circ\epsilon_l^w(xy)=f^\luc(xy)=0,\]
as we wanted to prove.

(iii)
Suppose now that $G$ is Abelian, let $f\in\luc(G)$ and $x,y\in G^{*l}$
with nets $(x_\alpha)$, $(y_\beta)$ converging to $x$ and $y$, respectively, in $G^\luc.$
 Clearly, $f\in \ruc(G)$, and so the last statement follows from the following observation
\[f^\ruc(x\squ y)=\lim_\beta \lim_\alpha f(x_\alpha y_\beta)
 =\lim_\beta \lim_\alpha f( y_\beta x_\alpha)=f^{\luc}(yx).\]
\end{proof}

\begin{remarks} \label{promisse}
 \begin{enumerate}
\item Although in some cases $\luc_\ast(G)\cap \ruc_\ast(G)= \wap_0(G)\oplus\C1$  (as for instance when $G=SL(2,\R)$ for, in that case, $\wap(G)=\cnaught\oplus \C1$), this equality does   not hold in general.
We see here that actually $\luc_\ast(G)=\ruc_\ast(G)$ is properly contained in $\wap_0(G)$.
Consider towards a counterexample the norm-closed unit ball $B_\infty$ of $L^\infty([0,1])$. With the weak$^*$-topology and pointwise
multiplication, $B_\infty$ is a compact commutative semitopological semigroup. This semigroup  $B_\infty$ contains a
homomorphic dense copy of the group of the integers $\Z,$
and it is a semigroup compactification of $\Z$ (see \cite{Bou}
and \cite{Pym2}).
Accordingly, there  is a continuous surjective homomorphism   $\pi:\Z^\wap\to B_\infty $.
Let now $K(\Z^\wap)$ be the minimal ideal of $\Z^\wap,$ which is a compact subgroup of $\Z^\wap$ contained in $\Z^{*w}\Z^{*w}$.
Since the minimal ideal  $K(B_\infty)$ of $B_\infty$ is easily seen to be  trivial, we see  that $\pi(K(\Z^\wap))=\{0\}.$

Let now  $f\in B_\infty$ be  the characteristic function of an open, proper, subset of $[0,1]$,  and let $x\in \Z^\wap$ be such that $ f=\pi(x)$. Then $xx\in \Z^{*w}\Z^{*w}$ but $xx\notin K(\Z^\wap)$, for otherwise $0=\pi(xx)=f^2=f$ which is absurd.

We may therefore pick $f\in \wap(\Z)$ with \[f^\wap(xx)=1
\quad\text{but}\quad f^\wap(K(\Z^\wap))=\{0\}.\]
The function $f$ is clearly in $\wap_0(\Z)$ since $K(\Z^\wap)$ is the support of the invariant mean
$\mu$ in $\wap(\Z)^*.$ But $f$ cannot be in $\luc_\ast(G)\cap\ruc_\ast(G)$,
since by Lemma \ref{12Other}, \[f^\luc(pp)=f^\wap\circ\epsilon_l^w(pp)=f^\wap(xx)=1,\]
where $p$ is a preimage in $G^{*l}$ of $x$.

This example shows also that $\wap_\ast(G)$ is properly contained in $\wap_0(G)\oplus\C1.$
\medskip

\item When $G$ is Abelian, Proposition \ref{abelian} implies that $\luc_a(G)$  is contained in $\luc_0(G).$ We will prove in Section 8, that
this is true for any locally compact group.
Note that the previous example shows that in general $\luc_a(G)$ is properly contained in $\luc_0(G).$
\medskip

\item Statement (iii) in Proposition \ref{abelian} does not hold when $G$
is not Abelian. We will show this in Example \ref{lwapvsrwap}, below, after identifying the $\luc_\ast(G)$-sets  in Theorem \ref{lwap}.
\end{enumerate}
\end{remarks}

We prove now this section's  key lemma.

 \begin{lemma}\label{general**}
  Let $G$ be a locally compact group,  $\A(G)$ be an admissible subalgebra of
$\CB(G)$.
If  $A\subseteq G$  is right (left) translation-compact and $f\in \A(G)$ is such that $f(G\setminus A)=\{0\}$, then $f^\A(G^\ast G^\ast)=\{0\}$ ($f^\A(G^\ast \squ G^\ast)=\{0\}$, respectively).
 \end{lemma}

\begin{proof}
  Let $p,q\in G^\ast$ and suppose that $f^\A(pq)\neq 0$.

  Choose nets $(s_\sigma)_{\sigma\in \Sigma } \subset G$ and $(t_\gamma)_{\gamma\in \Gamma}\subset  G$ such that \[\lim_\sigma \epsilon_\A(s_\sigma)=p\quad\text{ and}\quad \lim_\gamma\epsilon_\A( t_\gamma)=q\quad\text{ in}\quad G^\A.\]
  Since $pq=\lim_\sigma\lim_\gamma\epsilon_\A( s_\sigma t_\gamma)$, there are  $\varepsilon>0$ and $\sigma_0$ such that, for each $\sigma\geq \sigma_0$, $|\lim_\gamma f(s_\sigma t_\gamma)|>\varepsilon$.  Therefore, for each $\sigma\geq \sigma_0$, there is $\gamma(\sigma)$ such that for each $\gamma\geq \gamma(\sigma)$, we have $f(s_\sigma t_\gamma)\neq 0$. This implies that  $s_\sigma t_\gamma\in A$ for all such indices $\sigma$ and $\gamma$.

  Now, since $p\in G^\ast$,  the set $\{ s_\sigma \colon \sigma \geq \sigma_0\}$ is not relatively compact in $G$.  Since $A$ is right translation-compact, there is a finite subset $\{\sigma_1,\ldots,\sigma_n\}$ of $\{\sigma\in \Sigma\colon \sigma\geq \sigma_0\}$, such that $\bigcap_{i=1}^n s_{\sigma_{i}}^{-1} \,A$ is relatively compact. But this is impossible since
  \[ \left\{t_\gamma\colon \gamma>\gamma(\sigma_i),\; i=1,\ldots, n\right\} \subset \bigcap_{i=1}^n s_{\sigma_{i}}^{-1}\,A, \] and the former set is  not  relatively compact because $q\in G^\ast$.

  We conclude that $f^\A(pq)= 0$.
\end{proof}

The following theorem characterizes the approximable $\luc_\ast(G)$-interpolation sets with the right translation-compact sets, its mirror theorem for $\ruc_\ast(G)$ with left translation-compact sets follows also with the same arguments.

\begin{theorem}\label{lwap} Let $G$ be a noncompact, locally compact group and $T$ be a right uniformly
discrete with respect to some symmetric relatively compact neighbourhood $U$ of $e.$
Then the
following statements are equivalent.
\begin{enumerate}
\item Every function in $\luc(G)$ which is supported in $U T$ is in $\luc_\ast(G)$.
\item $T$ is  an $\luc_a(G)$-set.
\item $T$ is an $\luc_\ast(G)$-set.
\item $T$ is an approximable $\luc_\ast(G)$-interpolation set.
\item $T$ is  an approximable $\luc_a(G)$-interpolation set.
\item $UT$ is right translation-compact.
\end{enumerate}
\end{theorem}

\begin{proof}
We first prove that
(i),  (iii) and (vi) are equivalent.

(i)$\implies$(iii) Let $U_0$ be any open neighbourhood of $e$, and put $W=U_0\cap U.$ Let $V$ be   another  open neighbourhood of $e$ with $V^4\subseteq W$, and let $T_1\subseteq T$. Then, take
any right uniformly continuous function $\varphi$ with support
contained in $V^2$ and value $1$ on $\overline{ V}$  (note that $\overline V\subseteq V^2\subseteq U_0\cap U$), and  consider the function
$h:=\sum_{t\in T_1}\varphi_{t^{-1}}$. Then $h$ is supported in $UT$, and  $h\in \luc(G)$  by  \cite[Lemma 4.6 (ii)]{FG1} (where $h$ appears as $1_{T_1,\varphi}$).
Clearly, $h(VT_1)=\{1\}$ and $h(G\setminus (U_0T_1))=\{0\}$, and by assumption, $h\in\luc_\ast(G)$. Statement (iii) follows.

(iii)$\implies$(vi)
 Suppose that $UT$ is not right translation-compact.
 As seen in \cite[Lemma 4.11]{FG1}, following the argument used by Ruppert in \cite{rup} in the discrete
case,  it is possible to construct a subset $T_1$ of $T$ and a function $f$ with support contained in $UT_1$ such that  $f\notin\wap(G).$ We repeat here part of that construction  and see that, actually, $f\notin \luc_\ast(G)$.

Let  $L$ be a non-relatively compact subset of $G$ which contains no finite
subset $F$ for which $\bigcap_{b\in F}b^{-1}UT$ is relatively
compact. With no loss of generality, we may assume that $L$ is countable
and write it as $L=\{s_m:1\le m<\infty\}$. We may also assume that $L$ is right uniformly discrete
since for a fixed compact neighbourhood $W$ of $e$ we can find an infinite right $W$-uniformly discrete subset of $L$, and so $L$ may be taken as this subset.
Define inductively a sequence $(t_n)$ in $G$
as follows. Start with $s_1\in L$ and let $t_1\in s_1^{-1}T$.
Then $s_1t_1\in T$.
Suppose that $t_1,t_2,...,t_{n-1}$ have been
selected in $G$ such that $s_kt_l\in UT$ for every $1\le k\le l\le n-1$. Then
take \[t_n\in
\bigcap_{k\le n}s_k^{-1}UT.\] Note that this is possible because this set is not relatively compact.
The selection of the points $t_n$ is also made so that $(t_n)$ is not relatively compact.
To make sure of this, we fix again a compact neighbourhood $W$ of $e$ and take by Zorn's lemma a maximal
right $W$-uniformly subset $X$ of $s_1^{-1}UT.$ Then $s_1^{-1}UT\subseteq W^2X$ and so $\bigcap_{k\le n}s_k^{-1}UT \subseteq W^2X$ for each $n\in\N.$ Since each of these sets is not relatively compact,
we may select our points $t_n=w_nx_n$ so that the sequence $(x_n)$ is injective.

Note that the sequences $(s_m)$ and $(t_n)$ have  subnets $(s_\alpha)$ and $(t_\beta)$ with respective limits $s$ and $t$ in $G^{*l}$.
 For, if $(t_\alpha)=(w_\alpha x_\alpha)$ is a subnet of $(t_n)$ which converges to $t\in G^\luc,$
 then we may assume that $w_\alpha$ has a limit $w$ in $G$ since $W^2$ is compact, and $(x_\alpha)$ has a limit
 $x$ in $G^{*l}$
 since $(x_n)$ is an injective sequence in $X.$ The joint continuity property gives then $t=wx\in G^{*l}.$
The limit $s$ of $(s_\alpha)$ is clearly in $G^{*l}$ since $L$ is an infinite right uniformly discrete set.

Therefore,  we have $s_mt_n\in UT$ for every $1\le m\le n<\infty,$  and so for each $1\le m\le n$,
there exists $u_{nm}\in U$ such that $u_{nm}s_mt_n\in T.$
 This way we obtain a subset $T_1$ of $T$
 given by \[T_1=\{u_{nm}s_mt_n: 1\le m\le n<\infty\}\]
so that no $\luc$-function with value $1$ on $T_1$ is in $\luc_\ast(G).$ This will imply that $T$ cannot be an $\luc_\ast(G)$-set,
against statement (iii).
To check that $T_1$ is not an $\luc_\ast(G)$-set, note first that the net
$(u_{\alpha \beta})$ corresponding to the nets $(s_\alpha)$ and $(t_\beta)$  has a subnet,
  which we denote also by $(u_{\alpha\beta})$, with limit $u\in \overline U$ (the closure in $G$).
  Then,  for each fixed $\alpha,$ choosing $\beta$ sufficiently large, we obtain  $u_{\alpha\beta} s_\alpha t_\beta\in T_1$. Let then $f$ be any function in $\luc(G)$ with value $1$ on $T_1$. (Such functions indeed exist; in fact, since $T$ is an $\luc(G)$-set by Lemma \ref{udluc}, there is a neighbourhood $V$ of $e$ with $\overline{V}\subseteq U$ and  $f \in \luc(G)$ supported  in $U$ with $f( VT_1)=\{1\}$). Using the joint continuity property in
$G^\luc$,
we see  that
\begin{align*} f^\luc(uxy)=
\lim_{\alpha}\lim_{\beta} f(u_{\alpha\beta} s_{\alpha}t_{\beta})=1\end{align*}
Since $uxy\in G^{*l}G^{*l},$ we conclude that $f\notin\luc_\ast(G).$
So Statement (iii) fails, showing that (iii) implies (vi).

The implication (vi)$\implies$(i) follows from Lemma \ref{general**}.

We now prove that (ii), (iii), (iv) and (v) are equivalent.

It is obvious that (ii)$\implies$(iii). The equivalence  (iii) $\Longleftrightarrow$ (iv) was proved in Proposition \ref{apg=ag}.

Since it is obvious that (v) implies (ii), the theorem will be proved once we show that (iv) implies (v).

 (iv)$\implies $(v) Let $T$   be an approximable $\luc_\ast(G)$-interpolation set.
 Since, (iv) implies (iii),  and we already proved that (iii) implies (vi),  we know that $UT$ is right translation-compact.

 Let   $U_0$ be any neighbourhood  of $e$.
Since $T$ is an $\luc_\ast(G)$-set, there is an open neighbourhood $V$ of $e$ with $\overline{V}\subseteq U\cap U_0$ such
that for every   $T_1\subseteq T$  there is   $h\in \luc_\ast(G)$ with
$h(VT_1)=\{1\}$ and $h(G\setminus (U\cap U_0)T_1))=\{0\}$.
Since $(U\cap U_0)T_1$ is right translation-compact, Lemma \ref{general**}  implies  that  $h^\luc(G^{*l}G^{*l})=\{0\},$
thus $h$ is actually in $\luc_a(G)$. We have shown that $T$ is an $\luc_a(G)$-set. Now any bounded function $f\colon T\to \C$ has, by assumption, an extension $\tilde f\in \luc_\ast(G)$. The product $\tilde f\cdot h$ is then a function in $\luc_a(G)$ that extends $f$. We have  thus proved that $T$ is an approximable  $\luc_a(G)$-interpolation set.
\end{proof}

 If we allow $T$ in Theorem \ref{lwap} to be an $E$-set, then $T$ is an  $\A(G)$-set if and only if every $\uc(G)$-function supported in $UT$ is in $\A(G)$. This leads to a   characterization of $\wap_0(G)$-sets, as the one   obtained in  Theorem 4.22 of  \cite{FG1} where  some pretty deep  combinatorial results had to be invoked.

\begin{theorem}
   \label{wap=wap_0}
Let $G$ be a locally compact $E$-group and $\A(G)$ be an admissible subalgebra of $\wap(G)$.
Let  $T$ be a right  $U$-uniformly discrete $E$-set   for some symmetric relatively compact  neighbourhood $U$ of the identity, and consider the following statements.

 \begin{enumerate}
 \item Every function in $\uc(G)$ which is supported in $U T$ is in $\A(G)$.
   \item $T$ is  an $\A_0(G)$-set.
  \item $T$ is  an $\A(G)$-set.
   \item $T$ is  an approximable $\A(G)$-interpolation set.
   \item $T$ is  an approximable $\A_0(G)$-interpolation set.
     \item $UT$ is translation-compact.
  \end{enumerate}
   Then the four statements (ii)-(v) are equivalent, necessary for (i) and sufficient for  (vi).

  When  $\wap_\ast(G)\subseteq \A(G)\subseteq\wap(G)$, the six statements are equivalent.
  In particular, this is true when $\A(G)=\wap(G),$   $\wap_\ast(G)$ or $\wap_0(G)\oplus \C1$.
 \end{theorem}

\begin{proof}
The  proof of this Theorem is essentially the same as that of   Theorem \ref{lwap}. We will just point out the differences.

(i) $\implies$ (iii)  The only difference with the corresponding implication of Theorem \ref{lwap} is that one should invoke Statement (iii) of Lemma 4.6 of \cite{FG1} instead of   Statement (ii), to have $h\in\uc(G)$.

The implication (iii) $\implies$ (vi) is proved in
Lemma \ref{cor:wap2}.

To proof of the  equivalence of  statements (ii)-(v) in  Theorem \ref{lwap} works here if (iv)$\implies $(v) is slightly adapted.  The  corresponding proof in Theorem \ref{lwap} would produce   $h\in \A(G)$ with $h^{\A}(G^\ast G^\ast)= \{0\}$. Since by Lemma \ref{im}, the support of the
invariant mean $\mu\in \A(G)^\ast$ is contained in $G^\ast G^\ast$, this implies that $h\in \A_0(G)$.

 When $\A(G)=\wap_\ast(G)$, and hence  when $\A(G) $ is  any subalgebra of $\wap(G)$ containing $\wap_\ast(G)$, we apply
 \cite[Lemma 4.3]{FG1} where it is proven that (vi)  implies  (i).
All six statements  are then  equivalent.
   \end{proof}

Here is now an example of an $\ruc_\ast(G)$-set which is not an $\luc_\ast(G)$-set and hence it is not  a $\wap_0(G)$-set either.

 \begin{example}\label{lwapvsrwap}
 \emph{ Let $X=\{a_1,a_2,\ldots\}$ be a countable set and let $G=F(X)$ denote the free group on $X$. Define
 \[T=\left\{w\in G\colon w=a_1^{n_1}a_2^{n_2}\ldots a_k^{n_k} \mbox{ with }1\leq n_1<n_2<\cdots<n_k,\; k\in\N \right\}.\]
Then  $T$ is left translation-finite      but it is not right translation-finite. As a consequence,  $T$ is an $\ruc_\ast(G)$-set which is not an  $\luc_\ast(G)$  and so $\luc_\ast(G)\neq \ruc_\ast(G)$.
The set $T$ is not a $\wap(G)$-set either.}
\end{example}

\begin{proof}
We first show that $T$ is not right translation-finite.

Define for every $k\in \N$, $w_k=a_1a_2^{-k}a_1^{-1}\in G$ and
 $L=\{w_k\colon k\in \N\}$ and let
$\{w_{k_j}\colon 1\leq j\leq N\}$ be any finite subset of $L$.
For each $j$, $1\leq j\leq N$, and each  $n\in\N$,
 $w_{k_j} a_1 a_2^n=a_1a_2^{n-k_j}$. Therefore,

 \[a_1a_2^{n}=w_{k_1}^{-1}a_1a_2^{n-k_1}=w_{k_2}^{-1}a_1a_2^{n-k_2}
 =\cdots = w_{k_N}^{-1}a_1a_2^{n-k_N}\in w_{k_1}^{-1}T\cap w_{k_2}^{-1}T\cap\cdots\cap w_{k_N}^{-1}T,\] provided $n>k_j+1$ for $1\leq j \leq N.$

 Hence,
 $ w_{k_1}^{-1}T\cap w_{k_2}^{-1}T\cap\cdots\cap w_{k_N}^{-1}T$ is infinite and $T$ is not right translation-finite.

 We see now that  $T$ is a left t-set.
 Suppose there is $e\neq w\in G$ such that $T\cap Tw\ne\emptyset$, and pick  $n_{1,k}<\cdots<n_{j(k),k}$ and  $m_{1,k}<\cdots<m_{l(k),k},$ $k\in \N$ such that
 \[ a_{1}^{n_{1,k}}\cdots a_{j(k)}^{n_{j(k),k}}w=a_{1}^{m_{1,k}}\cdots a_{l(k)}^{m_{l(k),k}}\in Tw\cap T.\]
We assume that $n_{1,k}\neq m_{1,k};$ otherwise we start with $n_{2,k}$ and $m_{2,k}$.
 Then, there are no cancellations in the word
 \begin{align*}
   w&=a_{j(k)}^{-{n_{j(k),k}}}\cdots a_{2,k}^{-n_{2,k}} a_{1}^{(m_{1,k}-n_{1,k})} a_{2}^{m_{2,k}} \cdots a_{l(k)}^{m_{l(k),k}}
 \end{align*}
since $n_{1,k}\neq m_{1,k}$.  But that means that the integers $l(k)$ are all equal, and so are the integers $ m_{l(k),k}$ for all $k$. If we put $l(k)=l_0$ and $ m_{l(k),k}=m_0$, then we see that
\[Tw\cap T\subset \left\{w\in G\colon w=a_1^{n_1}a_2^{n_2}\ldots a_{l_0}^{n_{l_0}} \mbox{ with }1\leq n_1<n_2<\cdots<n_{l_0}=m_0 \right\},\] showing
that $Tw\cap T$ must be finite.

By Theorem \ref{lwap},  $T$ is   an $\ruc_\ast(G)$-set that is not an $\luc_\ast(G)$-set,  and  the characteristic function $1_{_{T}}$ of $T$ is in $\ruc_\ast(G)\setminus \luc_\ast(G)$. Also by Theorem \ref{wap=wap_0} (or, directly by \cite[Theorem 7]{rup} or \cite[Proposition 2.4]{C4}), $T$ is not a $\wap(G)$-set.

Since $G$ is discrete in this case, this example shows in particular that there are bounded functions on $G$ (for instance $1_{_{T}}$) which annihilate
 all $G^*\square G^*$ but do not annihilate $G^* G^*$ and vise-versa, where $G^*=\beta G\setminus G$.
 \end{proof}

\section{Veech's Theorem}\label{sec:veech}

In 1960, Ellis proved that $sx\ne x$ for every  $s\in G$, $s\ne e$ and $x\in \beta G$ (see \cite{E}).
 In 1977, Veech proved the theorem for $G^\luc$ with $G$ being a locally compact group (see \cite{V} or \cite{BJM}, and \cite{Ruppert} for  special cases).
In 1999, Pym simplified Veech's arguments (see \cite{Pym}). In all these references,  the proofs relied on  the Three Sets Lemma  stated below, originally due to Baker \cite{Ba}.

In this section,  we prove that  the points in the closure of right uniformly discrete $\A(G)$-sets satisfy Veech's property
in  the spectrum $G^\A$ of {\it any} unital left translation invariant $C^\ast$-subalgebra $\A(G)$ of $\CB(G)$ and for
 {\it any} topological group $G$.

 As indicated in the previous section, noting that  right uniformly discrete sets are $\luc(G)$-sets,  this may be regarded as a generalization of Veech's Theorem from the $\luc$-compactification $G^\luc$ of a locally compact group to the spectra $G^\A$ of these $C^\ast$-algebras.

Our  proof still relies  on the Three Sets Lemma but it is immediate
and neither the technical arguments needed by Veech in \cite{V}  or Ruppert in \cite{Ruppert} nor the Local Structure Theorem used by Pym in \cite{Pym} are necessary to obtain this generalization.

We recall first the Three Sets Lemma as presented by Pym in \cite[Lemma]{Pym}.

\begin{lemma} \label{3setslemma} Let $T_0$ be a subset of a given non-empty set $T,$ let $f:T_0\to T$ be an injective map such that
$f(t)\ne t$  for every $t\in T_0$. Then there is a partition $T=T_1\cup T_2\cup T_3$ such that
\[f(T_i\cap T_0)\cap T_i=\emptyset\quad \text{for each} \quad  i=1,2,3.\]
\end{lemma}

\begin{AVeech}\label{GeneralVeech}  Let $G$ be a topological group,  and $\A(G)$
be a unital left
invariant $C^\ast$-subalgebra of $\CB(G)$,
   and $T$ be an $\A(G)$-set which is right uniformly discrete.
   Then   $sx\ne x$ for every $x\in \overline T$ and $s\in G,$
$s\ne e$.
\end{AVeech}

\begin{proof}
Let $s\in G,$ $s\ne e$ and $x\in \overline T.$ Let $T$ be right uniformly discrete with respect to some neighbourhood
$U$ of the identity. Let
$V\subseteq U$ be a symmetric neighbourhood of $e$ with $s\notin  V$ and $s^{-1}V^2 s\subseteq U$, and let $T_0=\{t\in T:st\in VT\}$.
The case of $x\in \overline {T\setminus T_0}$ follows directly from the fact that $T$ is an $\A(G)$-set.
To see this, let $W$ be a neighbourhood of $e$ such that $\overline W\subseteq V$ and take $h\in \A(G)$ with
$h(G\setminus VT)=0$ and $h(WT)=1.$ Since $s{(T\setminus T_0)}\cap VT=\emptyset,$ this clearly implies that $h^\A(sx)=0$
and $h^\A(x)=1.$

Before  dealing with the case
 $x\in \overline{T_0}$, we note that for each $t\in T_0$ there exist unique $v_t\in V$ and $t_1\in T$ such that $st=v_tt_1.$ For, if $t_1$ and $t_2$ satisfy
$st=v_1t_1=v_2t_2$ for some $v_1,v_2\in V,$ then $t_1=t_2$ since $T$ is right $U$-uniformly discrete, and so $v_1=v_2.$
We may therefore define a function $f$ on $T_0$ by $st=v_tf(t)$, $v_t\in V$ and $f(t)\in T.$
Since $s\notin V,$ we see that $f(t)\ne t$.
Moreover, if $f(t_1)=f(t_2)$ for some $t_1$ and $t_2\in T_0,$  then $v_1st_1=v_2st_2$ for some $v_1,$ $v_2\in V$, and so we have
 $su_1t_1=su_2t_2$ for some
$u_1,$ $u_2\in U$ (since $s^{-1}V^2s\subseteq U$). Since $T$ is right $U$-uniformly discrete, we must have $t_1=t_2.$ In other words, the function $f:T_0\to T$
is injective.

 We apply  the Three Sets Lemma (Lemma \ref{3setslemma}) to $f$ and obtain  a  partition of $T$ into three sets
$T_1,$ $T_2$ and $T_3$ such that \[f(T_i\cap T_0)\cap T_i=\emptyset\quad\text{ for each}\quad i=1,2,3.\]
This latter property means precisely that $s(T_i\cap T_0)\cap VT_i=\emptyset$ (otherwise, $st=vt'$ for $v\in V$, $t\in T_i\cap T_0,$ $t'\in T_i$
means that $t\in T_i\cap T_0$ and $t'=f(t)\in T_i$).

Let now
 $x\in \overline{T_0}$. Then $x$ must belong to $\overline{T_i\cap T_0}$ for some $i=1,2,3$. The fact that $T$ is an $\A(G)$-set,
gives again a function $h\in\A(G)$ such that $h^\A$ separates $sx$ from $\overline{T_i\cap T_0}$ (and hence, from $x$) in $G^\A.$
\end{proof}

Veech's Theorem is then an easy consequence.

\begin{corollary}\label{Veech} Let $G$ be a locally compact group. Then $sx\ne x$ for every $x\in G^\luc$ and $s\in G,$ $s\ne e$.
\end{corollary}

\begin{proof} We only need to observe that any $x\in G^\luc$ can be seen in the closure of some right uniformly discrete set $T$.
To see this, fix a relatively compact neighbourhood $U$ of $e$ and consider first  a maximal right $U$-uniformly discrete subset $X$ of
 $G$.
Then $G=U^2X$, and so the joint continuity gives $G^\luc=\overline{U^2}\overline X.$ Thus, $x=ua$ for some $u\in \overline{U^2}$ and $a\in\overline X$. Then instead of $U$ and $X$, we take the relatively compact neighbourhood  $uUu^{-1}$ of $e$ and $T=uX$, respectively. Then $T$ is right $uUu^{-1}$-uniformly discrete and $x\in \overline{T}$. (This {\it astuce} was used by Pym first in \cite{Pym} and was useful also in
\cite{FP}.)
Since right uniformly discrete sets are $\luc$-sets, Theorem  \ref{udluc}, the corollary is immediate.
\end{proof}

\begin{remarks}  \begin{enumerate}
\item The reader may have noticed that while Theorem \ref{GeneralVeech} is proved for any topological group $G$, Corollary \ref{Veech} is proved only for locally compact groups. In fact, Veech's theorem may fail  dramatically at some points of $G^\luc$ when $G$ is not locally compact, as for example  when $G$ is  extremely amenable. A topological group $G$  is said to be extremely amenable when $G^\luc$ has a point which is left invariant, i.e., when there exists $x\in G^\luc$ such that $sx=x$ for every  $s\in G.$  Examples of extremely amenable groups, and further details  on this  interesting theory can be found, for instance, in  \cite{farahsole08,giordpest07,pest05}.

 \item Even if Veech's  theorem may fail,  the set of points at  which it fails is always   topologically small, see Corollary
\ref{failveech}, below.

\item Theorem \ref{GeneralVeech}  brings out another curious property of extremely amenable groups: if $G$ is such a group, the left invariant elements in $G^\luc$   cannot be  in the closure of any uniformly right discrete subset of $G$. This may be compared  with the locally compact case, when $G$ is locally compact, every point $x\in G^\luc\setminus G$ can be found in the closure of some  right uniformly discrete subset  of $G$, see  the proof of Corollary \ref{Veech}.
 \end{enumerate}
  \end{remarks}

 Since translation-compact sets exist in abundance in a noncompact locally compact group, and they are $\wap$-approximable interpolation sets,  provided they are $E$-sets as well, the main results obtained in \cite{BF},  \cite{FS} and \cite{F07}  follow easily.
 We re-state for instance the main result of \cite{BF}.
More properties using translation-compact sets will be developed in the rest of the paper.

\begin{corollary}\label{wap} Let $G$ be a noncompact, locally compact $SIN$-group.
Then the interior of the  set of points in $G^\wap\setminus G$
for which $xs\ne x$ and $sx\ne x$ whenever $s\ne e $ is dense  in $G^\wap\setminus G.$


\end{corollary}

\begin{proof}
 Let $\mathcal{V}=\left\{x\in G^\wap\colon sx\ne x \mbox{ whenever $s\ne e $}\right\}$.
   Observe that whenever $x\in \mathcal{V}$ then $xs\neq x$ as well,  since for every    $f\in \wap(G)$ the function $\check{f}$ defined on $G$ by $\check{f}(t)=f(t^{-1})$ is also in $\wap(G)$.

We prove that the interior of $\mathcal{V}\setminus G$ is dense in $G^\wap\setminus G$.

 Let $W_0\not\subset G$ be an open subset of $G^\wap$,  and take another open subset $W$ of  $G^\wap$  with $\overline{W}\subset W_0$ and $W\not\subset G$.  Fix  a relatively compact neighbourhood $U$ of $e$. The set $W\cap G$ is not relatively compact in $G,$ and so we can use  Example \ref{t-set} to find an infinite  right $U$-uniformly discrete subset $T\subset W\cap G$ such that $UT$ is translation-compact.
Since, by Lemma  \ref{lem:lucint}, these sets are $\wap(G)$-sets, Theorem \ref{GeneralVeech} implies that each point in $x\in \overline T$ satisfies $sx\ne x$ for every $s\in G,$ $s\ne e.$
Clearly then $sux\neq ux$  for every $u\in G$, $x\in \overline T$ and $s\in G,$ $s\ne e.$
Therefore $G\overline T\subset \mathcal{V}$.  Since $G\overline T\cap (W_0\setminus G)\neq \emptyset$ and  $G\overline T$ is open in $G^\wap$ by Theorem \ref{openlocals}, proved in the next section, the Corollary follows.
\end{proof}

\begin{corollary}\label{luc_a} Let $G$ be a noncompact, locally compact group. Then the set of points in $G^{\lwap}\setminus G$
for which $sx\ne x$ whenever $s\ne e $ has a dense interior in $G^{\lwap}\setminus G.$
\end{corollary}
\begin{proof} We can repeat the proof of Corollary \ref{wap} here. By  Theorem \ref{lwap}, the set $T$ in that proof is always an $\luc_\ast(G)$-set and   we do  need the $SIN$-property anymore. \end{proof}

Theorem \ref{GeneralVeech} has the following interesting corollary.
\begin{corollary}\label{failveech} Let $G$ be a topological group. Then
\begin{enumerate}
\item The set of points in $G^\luc\setminus G$ at which Veech's property holds is dense in $G^\luc\setminus G.$
 \item The set of left invariant points in $G^\luc\setminus G$ has an empty interior in $ G^\luc\setminus G.$
 \end{enumerate}
\end{corollary}

\begin{proof} If $G$ is totally bounded, the corollary is clear since $G^\luc=G^\ap$ (which is the Ra\u{i}kov completion of $G$).

In general, let $O$ and $P$ be  open sets in $G^\luc$ such that $\overline{P}
\subseteq O$ and $P\cap (G^\luc\setminus G)\ne\emptyset.$
Suppose first that $P\cap G$ is right
totally bounded (i.e., for every neighbourhood $U$ of $e$, there exists a finite subset $F$
of $G$ such that $P\cap G\subseteq UF).$ We claim that the points in $\overline{P\cap G}$ (closure in $G^\luc$)
are limits of right Cauchy nets. In other words, $\overline{P\cap G}$ is contained in a subsemigroup of $G^\luc$ which is isomorphic to  Weil right completion. So let $x$ be any of such points, and suppose that $x$ is the limit of some net  $ (x_\alpha)$ taken from the set $=P\cap G.$ Let $U$ be any open neighbourhood of $e$ in $G$
and choose another neighbourhood $V$ of $e$ such that $V^2\subseteq U.$
Since $P\cap G$ is assumed to be right totally bounded, we have $P\cap G\subseteq VF$ for some finite subset $F$ of $G.$ It follows that $x=\bar xa, $ where $x\in \overline V$ and $a\in F.$ In other words,
$x$ is eventually the limit of a net $(v_\alpha a)$ taken from  $Va$, and so  $x_\alpha x_\beta^{-1}=v_\alpha v_\beta^{-1}\in V^2\subseteq U$. Thus, $(x_\alpha)$ is a right Cauchy net in $G$, as wanted.

Consider then an accumulation point $a\in G^\luc$ of the net $(x_\alpha^{-1})$. Then \[f^\luc(xa)=\lim_\alpha\lim_\beta f(x_\alpha x_\beta^{-1} )=f(e)\] for every $f\in \luc(G)$, and so $xa=e$.
Since $x$ has a right inverse in $G^\luc$,
 this case is now clear. In fact, even more is true: the point $x$ is right cancelable in $G^\luc$
 in the sense that $yx\ne zx$ whenever $y\ne z$ in $G^\luc$ (not only in $G$).

If, otherwise, $P\cap G$ is not right totally bounded, then  we may construct by induction an infinite right uniformly discrete subset $T$ of $O_1\cap G$, and apply Theorem \ref{GeneralVeech}.
\end{proof}

\section{Pym's Theorem}\label{sec:pym}
Pym's  Local Structure Theorem, Theorem \ref{LST},
clarifies the structure of $G^\luc$. Its analog for $G^\wap$
was obtained in \cite[Corollary 1]{F07} although in this case the structure theorem only  applies at points that are in the closure of  right uniformly discrete t-sets. Such a restriction is however  necessary for,  as already observed without proof in \cite[page 170]{BP},  the Local Structure Theorem does not hold at points belonging to  the minimal ideal of $G^\wap$. This will be clarified in Section \ref{sec:sp}.
As with Veech's Theorem,  we will also see in Section \ref{sec:sp} that the Local Structure Theorem holds on a dense subset of $G^\A\setminus G$, for many admissible subalgebras $\A(G)$ of $\luc(G)$.

 The sets used in   Theorem \ref{LST} are $\luc(G)$-sets  (since they are right uniformly discrete)
 and the sets used in \cite{F07} are $\wap(G)$-sets (since they are right uniformly discrete t-sets, see Lemma  \ref{lem:lucint}).
 We observe therefore that, in both cases,  $\A(G)=\luc(G)$ and $\A(G)=\wap(G)$, the  Local Structure Theorem actually describes the neighbourhoods in $G^\A$  of closures of  right uniformly discrete $\A(G)$-sets.

          In this section we see that this a completely general feature of the $\A(G)$-property and we obtain an $\A$-Local Structure Theorem (Theorem \ref{locals}), that holds  in the spectrum $G^\A$ of 					
					any unital left tranlation invariant C$^*$-subalgebra $\A(G)$ of $\luc(G)$.

In \cite{BP}, Budak and Pym     extended the Local Structure Theorem to the Two-Point Local Structure Theorem and to the Compact-set Local Structure Theorem.  With $\A$-Veech's Theorem and the $\A$-Local Structure Theorem now at  hand, Budak's and Pym's steps can be easily followed so that both the $\A$-Two-Point Local Structure Theorem and the $\A$-compact-set Local Structure Theorem hold as well.

We end the section by  showing that $G^\A$ contains sets larger than $G$ on which the map $p\mapsto px$ is injective, for any $x\in G^\A$, thus extending Veech's Theorem to sets larger than $G$. This was obtained by Budak and Pym  \cite{BP} in $G^\luc$ for $\sigma$-compact groups on which the canonical map $G\to G^\ap$ is not surjective, and is done here for any metrizable locally compact group that contains an infinite $\ap(G)$-interpolation set that is an $E$-set   and any unital left tranlation invariant C$^*$-subalgebra $\A(G)$ of $\luc(G)$.

\medskip
\subsection{The Local Structure Theorem for $\A(G)$}

The necessity in following theorem is part of Pym's Local Structure Theorem. We  prove it here  for any unital left tranlation invariant $C^*$-subalgebra
of $\luc(G)$. The set $T$ does not even need to be right uniformly discrete, we only have to require it to be an   $\A(G)$-set.

\begin{theorem}\label{openlocals}
Let $G$ be a locally compact group, $\A(G)$ be a unital left tranlation invariant $C^*$-subalgebra
of $\luc(G)$ and $T$ be a subset of $G$. Then $T$ is an $\A(G)$-set if and only if
the set $ V\overline{T_1}$ (the closure in  $G^\A$) is  open in $G^\A$ for every subset $T_1$ of $T$ and every open set $V$ in $G.$
\end{theorem}

\begin{proof} Let $T$ be an $\A(G)$-set in $G$. We prove the claim for $T$, the proof is similar if $T_1$ is any subset of $T.$
Since any translation in $G^\A$ by  an element of $G$ is a homeomorphism, it is enough to assume that  $V$ is an open neighbourhood of $e.$
Let $vx\in V\overline T$ with $v\in V$ and $x\in\overline T.$
Choose another neighbourhood $V_0$ of $e$ such that $v\overline{V_0}\subseteq V.$ Since $T$ is an $\A(G)$-set, there exist a neighbourhood $W$ of $e$, with
$\overline{W}\subseteq V_0$,  and
$h\in \A(G)$ such that $\restr{h}{WT}=1$ and $\restr{h}{G\setminus V_0T}=0.$
Define $f\in\A(G)$ by $f(s)=h(v^{-1}s)$ and extend it to a continuous function $f^\A$ on $G^\A$.
Since $f(vT)=\{1\}$, we have $f^\A(vx)=1$ and so
 $vx\in {(f^\A)}^{-1}(]1/2,\infty[)\subseteq V\overline T.$ To see the inclusion, note that
if $y\in G^\A$ is such that $f^\A(y)>1/2$ and $(t_\alpha)$ is a net in $G$ converging to $y,$ then
\[1/2<f^\A(y)=\lim_\alpha f(t_\alpha)=\lim_\alpha h(v^{-1}t_\alpha),\] and so
the net $(v^{-1}t_\alpha)$ is eventually in $V_0T.$  Since $\overline{V_0} $ is compact in $G$, the joint continuity property implies that
$v^{-1}y\in \overline{V_0}\,\overline T,$
and so $y\in (v\overline{V_0})\,\overline T\subseteq V\overline T.$

Therefore, $vx\in {(f^\A)}^{-1}(]1/2,\infty[)\subset V\overline{T}$, and so $V\overline{T}$ is a neighbourhood of each of its points.

For the converse, suppose that $ V\overline{T_1}$ is  open in $G^\A$ for every subset $T_1$ of $T$ and every open set $V$ in $G.$ Let $U$ be any open neighbourhood of $e$. Choose arbitrarily a relatively compact neighbourhood   $W$ of $e$ such that $\overline{W}\subset U\cap V$. Let also $T_1$ be any subset of $T$.
Since $V\overline{T_1}$ is open in $G^\A$, the sets $\overline{WT_1}$ and $G^\A\setminus V\overline{T_1}$
are disjoint and closed in $G^\A,$ and so we can take a continuous function $H\colon G^\A\to \C$ that takes the value 1 on $\overline{WT_1}$ and vanishes on $G^\A\setminus V\overline{T_1}$. Then $\restr{H}{G}\in  \A(G)$ is the required function, showing that $T$ is an $\A(G)$-set.
\end{proof}

\begin{Alocal}\label{locals}
Let G be a locally compact group $G$, $\A(G)$ be a unital left tranlation invariant C$^*$-subalgebra
of $\luc(G)$, $U$ an open neighbourhood of $e$,  $T$ an $\A(G)$-set which is right $U$-uniformly discrete, and $T_1$ any subset of $T$.
Then $\overline {T_1}$ (the closure in  $G^\A$) is homeomorphic to
$\beta {T_1}$. Furthermore, for each open neighbourhood $V$ of $e$ in $ G$  with $\overline V \subseteq U,$ the set $ V\overline {T_1}$ is  open in $G^\A$ and homeomorphic to $V \times \beta {T_1}.$
\end{Alocal}

\begin{proof} As in the previous theorem, we prove the claim for $T$ as the proof is similar if $T_1$ is any subset of $T.$
Let $V$ be a fixed neighbourhood of $e$ such that $\overline V \subseteq U.$
That
$ V\overline T$ is  open in $G^\A$ follows from Theorem \ref{openlocals} and that the set  $\overline T$ is homeomorphic to $\beta T$  is a consequence of $T$ being  an $\A(G)$-interpolation set   (see Proposition \ref{apg=ag} and \cite[Lemma 2.3]{FG1}).

 For the rest of the claim, we use  the joint continuity property to extend the map $(u,t)\mapsto ut: V\times T\to VT$
to a continuous surjective map $\overline V\times \overline T\to \overline{VT}$. So   we obtain a continuous surjective map $\overline V\times \beta T\to \overline V\overline T.$
For this map to be a homeomorphism, we only need to check that it is injective.
Let $(u,x)\neq (v,y)$ be two distinct elements in $\overline{V}\times \overline{T}$.

Suppose first that $u\ne v$. Then  $uT\cap vT=\emptyset$ since $T$ is right $U$-uniformly discrete. Let then $W$ be a neighbourhood of $e$
such that $Wu\cup Wv\subseteq U$ and $Wu\cap Wv=\emptyset.$ Then   $uT\cup vT$ is right $W$-uniformly discrete, too. As  $uT$ and $vT$ are clearly $\A(G)$-sets, it follows from \cite[Proposition 5.2]{FG1} that $uT\cup vT$ is an $\A(G)$-set as well.  By \cite[Lemma 2.3]{FG2} (a standard property of the Stone-\v{C}ech compactification),  $u\overline T\cap v\overline T=\emptyset,$ and so $ux\ne vx.$

Suppose now  that $u=v$. Then $x\neq y$, and so it is clear that $ux\ne uy$ (if $h\in \A(G)$ separates $x$ and $y$, then the left translate ${}_{u^{-1}}h$ of $h$ separates $ux$ and $vy$).

Consequently, $\overline V\times \beta T$ and $\overline{VT}$ are homeomorphic, and so are  $V \times \beta T$ and $V\overline T$ since $V \times \beta T$ is open in $\overline V \times \beta T$.
\end{proof}

Since right uniformly discrete sets in $G$ are $\luc(G)$-sets, and since when $G$ is locally compact, each point $x\in G^\luc$ can be found in the closure of such sets as seen in the proof of Corollary \ref{Veech}, Pym's Local Structure Theorem \ref{LST} is immediate.

If we put $\mathcal N$ for a base at $e$ made of relatively compact
neighbouroods of $e$
and $\mathcal X$ for the family of all right uniformly discrete subsets of $G$, then the following consequence of Pym's Local Structure Theorem \ref{LST} is obtained.  Ko\c{c}ak and Strauss obtained a more general version   of this result  for  uniform spaces  in \cite{KS}.

\begin{corollary} \label{lucbase} For a locally compact group $G,$ the family
\[\mathcal N_\luc=\{V\overline T:V\in\mathcal N\;\text{and}\;T\in \mathcal X\}\] is a base for the topology in $G^\luc.$
\end{corollary}

\begin{proof}
If $O$ is any open subset in $G^\luc$ and $x\in O,$ then the joint continuity property provides a member $U$ of $\mathcal N$
and a closed neighbourhood  $P$ of $x$ in $G^\luc$ such that $UP\subseteq O.$ Let $V\in \mathcal N$ be symmetric and satisfy  that $V\overline{V^2}\subseteq U$. Pick then a maximal right $V$-uniformly discrete set
$T$ in $G\cap P.$ Then $x\in \overline{G\cap P}\subseteq \overline{ V^2}\overline T.$ As in the proof of Corollary \ref{Veech}, we can see
$x$ in $u\overline {T}$ for some $u\in \overline{V^2}$, where $uT$ is right $(uVu^{-1})$-uniformly discrete. By Theorem \ref{locals}, $V(u\overline T)$ is a member of the family $\mathcal N_\luc$. Since  $x\in Vu\overline T\subseteq V\overline{V^2}\;\overline T\subseteq UP\subset O,$ the claim follows.
\end{proof}

\begin{remarks}  \begin{enumerate}
\item As already mentioned, Corollary \ref{lucbase} does not hold in general. We shall return to this point in Section \ref{sec:sp}, where we show that
the Local Structure Theorem holds on a dense subset of $G^\A\setminus G$ for many admissible subalgebras $\A(G)$ of $\luc(G)$.

\item  The $\A$-analogues of the three main results obtained by Budak and Pym in \cite{BP}, namely
 the Two-Point Local Structure \cite[Lemma 3.8]{BP}, the Compact-Set Local Structure \cite[Theorem 3.9]{BP} as well as \cite[Theorem 3.10]{BP},
can also be proved by repeating  \emph{mutatis mutandis} their respective  proofs in \cite{BP} with $\A(G)$ instead of $\luc(G)$.
We only  state here the $\A$-Two Point Local Structure Theorem. The corresponding versions of  \cite[Theorem 3.9]{BP} and \cite[Theorem 3.10]{BP} follow the same pattern. \end{enumerate}
\end{remarks}

\begin{theorem}\label{twopoints-locals}
Let G be a noncompact locally compact group, $\A(G)$ be a unital left tranlation invariant C$^*$-subalgebra
of $\luc(G)$, $U$ an open neighbourhood of $e$, $T$ an $\A(G)$-set which is right $U$-uniformly discrete
and $x\in\overline T$.
Then, there exists an open neighbourhood $V$ of $e$ in $ G$ and a subset $P$ of $T$
  such that for every $s\in G,$
the set $ sV\overline P \cup V\overline P$ is  an open neighbourhood of $sx$ and $x$ in $G^\A$  and is  homeomorphic to $(sV\cup V) \times \beta P.$
\end{theorem}

\begin{proof}
 The direct proof provided in \cite{BP}, using the Three Sets Lemma as the main tool, may be applied  here to show the theorem. We deduce instead this theorem from $\A$-Veech's Theorem \ref{GeneralVeech}
 and $\A$-Local Structure Theorem \ref{locals}.

 If $s=e$, this is  the $\A$-Local Structure Theorem, Theorem \ref{locals}.  If  $s\ne e$, then by $\A$-Veech's Theorem \ref{GeneralVeech},  $sx\ne x,$ and so we may separate $sx$ and $x$ by two basic open sets in $G^\A.$
  By Theorem \ref{locals}, these open sets may be taken of the form $sV\overline P$ and $V\overline P$ for some neighbourhood $V$ of $e$.
  Since $sV\overline P$ and $V\overline P$ are disjoint, the theorem follows.
   \end{proof}

\subsection{Injectivity points} \label{injectivity}
 In \cite[Section 4]{BP},  Budak and Pym
introduced,  for  each $E\subseteq G^\luc,$ the set \[\mathrm{Inj}(E) = \{x \in G^\luc: y\mapsto yx \;\text{is injective
on}\; E\}.\] With this terminology,  Veech's Theorem asserts  that $\mathrm{Inj} (G) = G^\luc$. In \cite[Theorem 4.1]{BP}, they proved that if $G$ is a locally compact, $\sigma$-compact group such that the
natural homomorphism $\phi: G \to G^\ap$ is not surjective, then there is an open  set $U$ in
$G^\luc$ which contains properly  $G$ and satisfies $\mathrm{Inj} (U) = \mathrm{Inj}(G)=G^\luc.$

The key tool used to reach this conclusion is a homomorphism   $\phi\colon G \to M$, with $M$ as a compact metric group,  that maps a countable right uniformly discrete set $X$ onto a sequence that converges to a point of $M\setminus \phi(G)$ (this can hold even if  $G^\ap$ is rather small, but  definitely  not trivial).
 The fact that $X$ is an $\luc(G)$-set allows the use of the Local Structure Theorem,
and $\sigma$-compactness  of $G$  allows the use of the so-called    \emph{butterfly lemma}, that belongs to  Stone-\v{C}ech compactifications of countable discrete spaces. These two powerful tools  not being available in smaller compactifications, we could not extend this approach to algebras such as $\wap(G)$ or $\B(G)$. We do  get analogous results on these algebras under the condition that $G^\ap$ is somewhat larger,  namely when $G$ contains at least one infinite $\ap(G)$-interpolation set. Any second countable locally compact group with uncountably many nonequivalent finite dimensional representations contains such a set, \cite{galihern04}.
This will be achieved in  Theorem \ref{injA}, but some preparation is needed first.

We shall begin by   observing  that approximate interpolation is always possible in a neighbourhood of an interpolation set, this is necessary since we are after an open set and our interpolation sets are always discrete.  We shall then have to recall some known facts on $\ap(G)$-interpolation sets (also known as $I_0$-sets), our main tool in Theorem \ref{injA}.  We will be ready after  a technical topological lemma that will help us to   relate the $\A(G)$-compactification with the $\ap(G)$-compatctification around $\ap(G)$-interpolation sets.

The approximate interpolation lemma  that follows is  a simplified  extension of Theorem  2 of \cite{kalt93} (see also \cite[Theorem 7]{rams96}  and \cite[Theorem 3.2.5]{GH}).

\begin{lemma}\label{apprint}
  Consider a metrizable locally compact  group $G$ and a $C^\ast$-subalgebra  $\A(G)$  of $\luc(G)$. If $T\subset G$ is an $\A(G)$-interpolation set, then for every $\varepsilon>0$ there is a neighbourhood $V=V(T,\varepsilon)$ of the identity such that for every  $f\colon T\to \T$ there is $\phi\in \A(G)$ such that $\left|\phi(ut)-f(t)\right|<\varepsilon$ for every $t\in T$ and every $u\in V$.
\end{lemma}

\begin{proof}
  Fix, to begin with, a base of neighbourhoods of the identity $\{V_n\colon n<\omega\}$ with $V_{n+1}\subset V_n$ for every $n$.

   Define for each $n$ and $\varepsilon$, the sets
  \begin{align*} \mathcal{A}(n,\varepsilon)= \{\phi\in \A(G)\colon \|\phi\|\leq 1,\quad \phi(T)\subset \T &\mbox{ and } |\phi(ux)-\phi(x)|\le\frac{\varepsilon}{2}\\& \mbox{ for all } u \in V_n\mbox{ and } x\in G\}.\end{align*}
  We shall denote the restriction of   these sets to $T$ as   \[\restr{\mathcal{A}(n,\varepsilon)}{T}= \left\{\restr{\phi}{T   }\colon \phi\in \A(n,\varepsilon)\right\}.\]
The   $\A(G)$-interpolation property for $T$ implies that any function in $\ell_\infty(T)$ extends
to a function in $\A(G)$ with the same norm (see \cite[Lemma 2.3 (iii)]{FG1}). Therefore, using also that $\A(G)\subseteq \luc(G),$ we see that
     \[ \T^T=\bigcup_n \restr{\A(n,\varepsilon)}{T}.\]

Now, regarded as subsets of $\ell_\infty(G),$ the sets $\mathcal{A}(n,\varepsilon)$
are weak$^*$-compact. Hence, the sets $\restr{\mathcal A(n,\varepsilon)}{T}$, being restrictions of the  sets
 $\mathcal{A}(n,\varepsilon)$, are closed in $\T^T$.
 The Baire category theorem provides then $n_0$ such that
the set     $\restr{\A(n_0,\varepsilon)}{T}$
 has nonempty interior. Since $\T^T$ is a (multiplicative) compact topological group  there is a finite set $\{ \phi_j\colon 1\leq j\leq N  \}\subset \A(G)$ such that \begin{equation} \label{TT} \T^T=\bigcup_{j=1}^N\restr{\phi_j}{T}\restr{\A(n_0,\varepsilon)}{T}.\end{equation}

 Since $\A(G)\subset \luc(G)$, there are neighbourhoods $V_{n(j)}$, $1\leq j \leq N$ such that
 \[\left|\phi_j(ux)-\phi_j(x)\right|<\frac{\varepsilon}{2}, \mbox{ for all  $u\in V_{n(j)}$ and $x\in G$.}\]

Take now $V=V_n$ with   $n\ge\max\{n_0,n(1),\ldots,n(N)\}$,
and consider an arbitrary function $f\colon T\to \T$.
By \eqref{TT} one  can then find $j$, with $1\leq j \leq N$ and $\phi_0\in \A(n_0,\varepsilon)$ such that
 \[ f=\restr{\phi_j}{T}\cdot \restr{\phi_0}{T}.\]
 Taking $\phi=\phi_j\cdot \phi_0$,   the desired approximation is obtained as follows
  \begin{align*}
    \left|\phi(ut)-f(t) \right|&\leq \left|\phi_{j}(ut)\phi_0(ut)-\phi_{j}(ut)\phi_0(t) \right|+\left|\phi_{j}(ut)\phi_0(t) -f(t) \right|\\
    &\leq   \left|\phi_0(ut)-\phi_0(t) \right|+\left|\phi_{j}(ut) -\phi_{j}(t) \right|< \varepsilon
  \end{align*}
  for every $u\in V$ and every $t\in T.$
 \end{proof}

 Now, the general properties of $\ap(G)$-interpolation sets.

\begin{lemma}\label{lem:I0}
Let $G$ be a metrizable locally compact group and let $T\subseteq G$ be an $\A(G)$-interpolation set. Then
 \begin{enumerate}[label=\textup{(\roman{enumi})}]
   \item \label{one}            $\epsilon_{\ap}$ is injective on $T.$
       \item \label{three}$\epsilon_{\ap}(T)$ is discrete.
       	\item \label{four} $\epsilon_{\ap}(T)$ is closed in $\epsilon_{\ap}(G)$.
	 \end{enumerate}
\end{lemma}

\begin{proof}
  Statements  \ref{one} and \ref{three}
   follow easily from the fact that \emph{any} bounded complex-valued function on $T$ extends to a function in $\ap(G)$.

   Statement \ref{four}  is well-known in the Abelian case. It has been proved by M\'ela \cite{mela68}, Ramsey  \cite{rams80} and Ryll-Nardzewski \cite{ryll64b}.
The proof given in \cite[Theorem 3.5.1]{GH}, based on  M\'ela's [loc.cit.], can be easily adapted to the  noncommutative case. As a   substitute of the dual group $\widehat{G}$  we may use a set
    of diagonal matrix coefficients of representations of $G$.
       Recall that for a given unitary representation $\pi \colon G\to \mathcal{U}({\mathbb{H}})$ on a Hilbert space $\mathbb{H}$, the  diagonal matrix coefficient associated to $\pi$ and $\xi \in \mathbb{H}$ is the function $f_{\pi,\xi}\colon G\to \C $ given by $f_{\pi,\xi}(s)=\langle \pi(s)\xi,\xi\rangle$.
       Any positive-definite function on $G$ is actually of this form.
        Denote by $\mathcal{P}_n (G)$ and $\mathcal{P}_f (G)$
      the sets of, respectively,  $n$-dimensional and finite dimensional diagonal matrix coefficients. In the topology of pointwise convergence $\mathcal{P}_n(G)$ is $\sigma$-compact, this can be deduced   from Theorem 2.1 in \cite{galihern04} and the continuity of the map $(\pi,\xi)\mapsto \phi_{\pi,\xi}$ defined on $Hom(G,\mathcal{U}(\C^m))\times \mathbb{B}_m$ ($\mathbb{B}_m$ stands for the  unit ball of $\C^m$).
        The proof of \cite[Theorem 3.5.1]{GH} can then be repeated replacing $\widehat{G}$ by $\mathcal{P}_f(G)$ and taking into account that every almost periodic function can be uniformly approximated by a linear combination of functions in  $\mathcal{P}_f(G)$,  \cite[Theorem 16.2.1]{dixm77}.
       \end{proof}

\begin{lemma}\label{lem:top}
Let $X_1$, $X_2$ and $X_3$ be three topological spaces, let $T\subset X_1$ and let $f_{ij} \colon X_i\to X_j$ with $i<j$ be three continuous maps such that the following diagram commutes:
 \begin{equation}\label{comm1} \xymatrix{
X_1 \ar[r]^{f_{13}} \ar[d]^{f_{12}} & X_3 \\
X_2 \ar[ur]_{f_{23}}. &}\end{equation}
 If $\restr{f_{13}}{T}$ is injective and $f_{13}(T)$ is   discrete in the relative topology of $f_{13}(X_1)$, then
\[ f_{12}(T)=\overline{f_{12}(T)}\cap f_{23}^{-1}\left(f_{13}(T)\right).\]
\end{lemma}

\begin{proof}
Since the diagram commutes, it is obvious that $ f_{12}(T)\subseteq\overline{f_{12}(T)}\cap f_{23}^{-1}\left(f_{13}(T)\right)$.

Let $p\in \overline{f_{12}(T)}\cap f_{23}^{-1}\left(f_{13}(T)\right)$  and let $t_p\in T $ such that $f_{23}(p)=f_{13}(t_p)$. Suppose  that $p\neq f_{12}(t_p)$. Then
\[ p\in \overline{f_{12}(T)}=\overline{f_{12}(T)\setminus \{f_{12}(t_p)\}}\subseteq \overline{f_{12}(T\setminus \{t_p)\}},\]
and hence \begin{align*}
f_{13}(t_p)&=f_{23}(p)\in f_{23}\left(\overline{f_{12}(T\setminus \{t_p\})}\right)\\
&\subseteq \overline{f_{23}\left(f_{12}(T\setminus \{t_p\})\right)}=\overline{f_{13}(T\setminus \{t_p\})}.
\end{align*}
Since $f_{13}$ is injective on $T,$ this contradicts the  discreteness of $f_{13}(T)$.
We conclude that $ p= f_{12}(t_p)$, and hence $p\in
f_{12}(T)$.
\end{proof}

\begin{theorem}\label{injA}
Let G be a metrizable locally compact group containing an infinite   $\ap(G)$-interpolation set $T$, let $\A(G)$ be an admissible algebra with $\ap(G) \subseteq  \A(G)\subseteq \luc(G)$, and suppose that $\mathrm{Inj}(G)=G^\A.$
Then there is an open neighbourhood $V$ of $e$
such that $\mathrm{Inj}(\epsilon_\A(G)\cup \epsilon_\A(V)\overline{\epsilon_\A(T)})=G^\A.$
\end{theorem}

\begin{proof}
Since we will be mapping $G$ into both $G^\ap$ and $G^\A$, and these mappings may not be 1-1 (especially the former one), we will exceptionally  make  in this proof explicit use of
the canonical morphisms into the compactifications   $\epsilon_{\ap}\colon G\to G^\ap$ and $\epsilon_\A\colon G\to G^\A$.
 We will also need  $\epsilon_\ap^{\A}\colon G^\A\to
G^\ap$, the semigroup homomorphism dual to the inclusion map  $\ap(G)\hookrightarrow \A(G)$.  The
  following diagram  \begin{equation}\label{comm2}
\xymatrix{
G \ar[r]^{\epsilon_{\ap}} \ar[d]^{\epsilon_\A} & G^\ap \\
 G^\A \ar[ur]_{\epsilon_\ap^{\A}} &}\end{equation}
 then commutes.

We apply now  Lemma \ref{apprint} to $T$ and $\varepsilon<1/2$ to obtain a neighbourhood  $V$ of the identity such that
  every function $f\in \T^T$ can be approximated by an almost periodic function $\phi$ in such  a way that $|f(t)-\phi(vt)|<\varepsilon$ for every $v\in V$ and $t\in T$. Moreover, since $T$ is an $\luc(G)$-interpolation set in a metrizable group, it is necessarily right uniformly discrete by \cite[Theorem 4.9]{FG1}, and so $V$ may be chosen such that $T$ is right $V$-uniformly discrete.

 Let   $x\in G^\A$  and $p_0,q_0\in \epsilon_\A(G)\cup \epsilon_\A(V)\overline{\epsilon_\A(T)}$  be such that $p_0x=q_0x$. Since $G^\ap$ is a group we have  $\epsilon_\ap^{\A}(p_0)=\epsilon_\ap^{\A}(q_0)$.

\textbf{Claim:} \emph{If $p_0=\epsilon_\A(s)$ with $s\in G$,  then $q_0$ must be also in $\epsilon_\A(G).$}

To see this, let $q_0=\epsilon_\A(v) q$ with $v\in V$ and $q\in \overline{\epsilon_\A(T)}$. Then we have

\[\epsilon_{\ap}(s)=\epsilon_\ap^{\A}(\epsilon_\A(s))=\epsilon_\ap^{\A}(p_0)=\epsilon_\ap^{\A}(q_0)=\epsilon_\ap^{\A} (\epsilon_\A(v) q)=\epsilon_{\ap}(v)\epsilon_\ap^{\A}(q),\]
and so $\epsilon_\ap^{\A}(q)\in \epsilon_{\ap}(G)\cap \overline{\epsilon_{\ap}(T)}$. By property \ref{four}  of  $\ap(G)$-interpolation sets in Lemma \ref{lem:I0},
we see that
$\epsilon_\ap^{\A}(q)\in \epsilon_{\ap}(T)$, i.e, $q\in (\epsilon_\ap^{\A})^{-1}\left(\epsilon_{\ap}(T)\right).$

 Using properties  \ref{one} and  \ref{three}  of  Lemma \ref{lem:I0},  we  see that  Diagram \eqref{comm2} satisfies the properties of Diagram  \eqref{comm1}.  Lemma \ref{lem:top} can then be applied to obtain that $\epsilon_\A(T)=\overline{\epsilon_\A(T)}\cap (\epsilon_\ap^{\A})^{-1}\left(\epsilon_{\ap}(T)\right)$. We obtain thus that $q\in \epsilon_\A(T)$, and the claim is proved.\medskip

The proof now follows depending on whether $p_0$ or $q_0\in \epsilon_\A(G)$.

\textbf{Case I:} \emph{
Either $p_0$ or $q_0\in \epsilon_\A(G)$}

Note that, by the above Claim, this implies that both $p_0$ \emph{and} $q_0$ are in $\epsilon_\A(G)$.

Then, the  assumption $\mathrm{Inj}(G)=G^\A$ implies that $p_0=q_0$.

\textbf{Case II:} \emph{Both $p_0$ and $q_0$ are in $\epsilon_\A(V)\overline{\epsilon_\A(T)}$.}

 Let $p,q\in \overline{\epsilon_\A(T)}$ and  $u,v\in V$ be such that $p_0=\epsilon_\A(u)p$ and $q_0=\epsilon_\A(v)q$.

 If $p=q$, then $\epsilon_\A(u)$ and $\epsilon_\A(v)$ must be also equal. Otherwise,
   again the assumption $\mathrm{Inj}(G)=G^\A$ gives  \[p_0x=\epsilon_\A(u)(px))\ne \epsilon_\A(v)(px)=\epsilon_\A(v)(qx)=q_0x,\]
 whence a contradiction is derived.

Suppose finally that  $p\neq q$.
There is then  $T_1\subset T$ such that $p\in \overline{\epsilon_\A(T_1)}$ and
 $q\in \overline{\epsilon_\A(T\setminus T_1)}$.
Consider a function $f\colon T\to \T$ with $f(t)= 1$ for all $t\in T_1$ and $f(t)=-1$ for $t\in T\setminus T_1$. By Lemma \ref{apprint} there is $\phi\in \ap(G)$ such that \[|\phi(ut)-1|<\varepsilon\quad\text{ for all}\quad t\in T_1\quad\text{ and}\quad |\phi(vt)-1|\geq 2-\varepsilon\quad\text{ for all}\quad t\in T\setminus T_1.\]
 Accordingly, \[|\phi^\ap\left(\epsilon_\ap^{\A}(\epsilon_\A(u)p)\right)-1|\le \epsilon\quad \text{and}\quad  |\phi^\ap\left(\epsilon_\ap^{\A}(\epsilon_\A(v)q)\right)-1|\ge 2-\epsilon,\]
 and so
 \[\phi^\ap(\epsilon_\ap^{\A}(p_0))=\phi^\ap(\epsilon_\ap^{\A}(\epsilon_\A(u)p))\ne \phi^\ap(\epsilon_\ap^{\A}(\epsilon_A(v)q))=\phi^\ap(\epsilon_\ap^{\A}(q_0)). \]
Thus, $\epsilon_\ap^{\A}(p_0)$ and $\epsilon_\ap^{\A}(q_0)$ are distinct points in $G^\ap$, and so are the points
  $\epsilon_\ap^{\A}(p_0)\epsilon_\ap^{\A}(x)$ and $\epsilon_\ap^{\A}(q_0)\epsilon_\ap^{\A}(x)$ in $G^\ap.$
  The points $p_0x$ and $q_0x$ must then also be distinct in $G^\A$, which is a contradiction. Hence, $p_0=q_0,$   as required.
We conclude that ${\displaystyle x\in \mathrm{Inj}\left(\epsilon_\A(G)\cup \epsilon_\A(V)\overline{\epsilon_\A(T)}\right)}$.
      \end{proof}
			
\begin{corollary}\label{cor:InjA} Let $G$ be as in Theorem \ref{injA} with  $\mathrm{Inj}( G)=G^\A$
and $\cnaught\oplus\ap(G)\subseteq\A(G)\subseteq \luc(G)$, and let $T$
  be an infinite  $\ap(G)$-interpolation set in $G$ that is in addition an $\A(G)$-set. Then there is  an open and dense subset  $D\subseteq G^\A$ containing properly $G$  such that $\mathrm{Inj}(D)=G^\A.$
\end{corollary}

\begin{proof}
  If we apply Theorem \ref{openlocals} to $\A(G)$ and Lemma \ref{apprint} to $\ap(G)$, we obtain a neighbourhood $V$ of $e$
	for which $T$ has the properties required in the proof of  Theorem \ref{injA}
	 (i.e., $T$ is ${\displaystyle V}$-right uniformly discrete and has the $\ap(G)$-approximation property)
	and such that $V\overline{T}$ is open in $G^\A$. Put $D=\epsilon_\A(G)\cup \epsilon_\A(V)\overline{\epsilon_\A(T)}$, $D$ is clearly dense in $G^\A$. The proof of Theorem \ref{injA} shows that $\mathrm{Inj}(D)=G^\A$. The condition  $\cnaught\subset \A(G)$ makes, by Lemma \ref{homeo}, $\epsilon_\A(G)$ open in $G^\A$, and so is $D.$
		\end{proof}

When $\A= \luc(G)$, this Corollary applies to every locally compact group containing infinite $\ap(G)$-interpolation sets

\begin{corollary}
\label{cor:inj-luc}
 If $G$ is a  metrizable locally compact group containing an infinite  $\ap(G)$-interpolation set,    then
  there is an open and dense subset $D$ of $G^\luc$ that contains properly  $G$ such that  $\mathrm{Inj} (D) = G^\luc.$
\end{corollary}

\begin{proof}
    This follows directly from Corollary \ref{cor:InjA} using Veech's Theorem (deduced in Corollary \ref{Veech}).
	It suffices to recall here that  $\ap(G)$-interpolation sets are also $\luc(G)$-interpolation sets, and so they are $\luc(G)$-sets by \cite[Proposition 3.3]{FG1}. 		
\end{proof}

 \begin{corollary}\label{cor:injE}
Let $G$ be a metrizable locally compact group containing an infinite   $\ap(G)$-interpolation set $T_0$ that is an $E$-set. Let $\A(G)$ be an admissible algebra with $\wap(G) \subseteq  \A(G)\subseteq \luc(G)$, and suppose that $\mathrm{Inj}( G)=G^\A.$
Then there is an open and dense subset  $D\subseteq G^\A$  that contains properly  $G$ such that $\mathrm{Inj}(D)=G^\A.$
  \end{corollary}

  \begin{proof}
  Let $V$ be the neighbourhood  obtained in the proof of Corollary \ref{cor:InjA}.
  The set $V$ may be chosen so that $T_0$ is $V$-right uniformly discrete, \cite[Theorem 4.9]{FG1}. So we can apply the construction of Example \ref{t-set}
   to obtain  $T\subseteq T_0$ such that $VT$ is in addition translation-compact. Hence $T$ is an approximable $\wap(G)$-interpolation set
	by Lemma \ref{lem:lucint}. It follows that  $T$ is an $\A(G)$-set for any algebra with $\wap(G)\subset \A(G)$. We can then apply Corollary \ref{cor:InjA}. \end{proof}

\begin{corollary}
 Let $G$ be a  metrizable locally compact Abelian group and let  $\A(G)$ be an admissible algebra with $\wap(G) \subseteq  \A(G)\subseteq \luc(G)$. Then
  there is an open and dense  subset $D$ of $G^\A$ that contains properly  $G$ such that  $\mathrm{Inj} (D) = G^\A.$
\end{corollary}
\begin{proof}
Locally compact Abelian groups are   $SIN$-groups and  contain infinite  $\ap(G)$-interpolation sets (plenty of them, actually, see for example \cite[Corollary 4.8]{galihern99}). Moreover   $\wap$-Veech's Theorem holds on them (as they are $MAP$-groups, see \cite[Theorem 1]{BF}). The corollary then follows from  Corollary \ref{cor:injE}.
\end{proof}

\section{Translation-compact sets and strongly prime points}\label{sec:sp}

Throughout the section, $G$ is a noncompact locally compact group, $\A(G)$ is an admissible  C$^\ast-$subalgebra of $\CB(G)$, and as before $G^\A$ is the $\A$-compactification of $G$  and  $G^*=G^{\A}\setminus \epsilon_\A(G)$. When $\A$ separates points and closed sets of $G$ (and this happens  whenever $\cnaught\subset \A(G)$), $G^\ast$ is closed in $G^\A$ and, therefore, both
$G^*G^*$ and $\overline{G^*G^*}$ are ideals in each of the semigroups
$G^{\A}$ and $G^*$, see Lemma \ref{homeo}. We say that a point $p\in G^{\A}$ is {\it prime} if
$p\notin G^*G^*$, and {\it strongly prime} if
$p\notin\overline{G^*G^*}$.

A combinatorial characterization of the prime points is still not known even in $\beta \Z$ when $\Z$ is the discrete additive group of the integers
 and $\A(\Z)=\ell_\infty(\Z)$.
This is related to the open question, known as Mary Ellen Rudin's question,  whether every point in $\Z^*=\beta \Z\setminus \Z$ belongs to some maximal principal left ideal $\beta\Z x$ of $\Z^*$, see \cite{HMS} or \cite[Question 33]{HM}.
In fact, it is not difficult to check that $\beta \Z x$ is a maximal principal left ideal whenever $x$
is prime, but the converse  is still not known.

In \cite{FP}, however,  a combinatorial characterization of the strongly prime points in $\beta G$
was obtained using right translation-finite sets in $G$.

In this section,
 we obtain a complete characterization of strongly prime points, in $G^{\luc}$ and in  $G^\wap$ for $SIN$-groups, in terms of translation-compact sets. This is then applied to prove that the union of right-translation-compact sets is translation-compact.
\subsection{Strongly prime points and translation-compact sets}

We start with our basic description of strongly prime points.

\begin{theorem}\label{sprime} Let $G$ be a locally compact group,  $\A(G)$ be an admissible subalgebra of
$\luc(G)$
and $T$ be a subset of $G.$  Consider the following statements.
 \begin{enumerate}
\item
$\overline T\cap G^*G^*=\emptyset.$
\item $KT$ is right translation-compact for every relatively compact subset $K$ of $G.$
\item For some  open relatively compact subset $V$ of $G$, $VT$ is  right translation-compact
and $T$ is an $\A(G)$-set.
\item For some  open relatively compact subset $V$ of $G$,  $VT$ is  right translation-compact  and $V\overline{T_1}$  is open in $G^\A $ for every $T_1\subseteq T.$
\item
$\overline T\cap \overline{G^*G^*}=\emptyset.$
\end{enumerate}
Then $(iii)\Longleftrightarrow(iv)\Longrightarrow (v) \Longrightarrow$ (i).
 If $C_0(G)\subseteq \A(G),$ then (i)$\Longrightarrow$ (ii).
 \end{theorem}

\begin{proof}
(iii)$\Longleftrightarrow$(iv). This follows from Theorem \ref{openlocals}.

(iv)$\Longrightarrow$(v).
 Choose, according to the definition of $\A(G)$-sets, a neighbourhood   $W$  with  $\overline{W}\subset V$ and $h\in \A(G)$ with $h(G\setminus VT)=\{0\}$ and $h(WT)=1$. Since $VT$ is also assumed to be right-translation-compact,  $h^\A(G^\ast G^\ast)=\{0\}$ by Lemma \ref{general**}. Therefore,
  $W\overline T\cap  G^\ast G^\ast=\emptyset$.  In particular, $\overline{T}\cap  G^\ast G^\ast=\emptyset$.
  Since $V\subset G$ and $G^*$ is invariant, this yields clearly  $V\overline T\cap  G^\ast G^\ast=\emptyset$. Since $V\overline T$
  is assumed to be open in $G^\A$,  Statement (v) follows.

(i)$\Longrightarrow$(ii). Suppose now that $C_0(G)\subseteq \A(G).$ Let $T\subseteq G$ and suppose that $KT$ is not right translation-compact for some relatively compact subset $K$
in $G$. Then pick a non-relatively compact
subset $L$ of $G$ such that $\bigcap_{b\in F}b^{-1}KT$ is not relatively compact in $G$ for
any finite subset $F\subset L$. As a subset of $G^\A$ (i.e., the set $\epsilon_A(\bigcap_{b\in F}b^{-1}KT)$) is not relatively compact either since $\epsilon_\A$ is a homeomorphism.
Therefore,
\[
\bigcap_{b\in F}b^{-1}\;\overline{ K\, T}\cap G^*\ne\emptyset\]
 for every finite subset $F\subset L$.

By the finite intersection property, it follows that there exists $y\in G^*$ such that $by\in \overline{ K T}$ for every $b\in L$. By
continuity of the shift $x\mapsto xy:G^{\A}\to G^{\A},$ and since $L$ is not relatively compact, there exists $x\in G^*\cap\overline L$ such that
$xy\in \overline{K T}$. By the joint continuity property,  we see that $\overline{KT}=\overline K\;\overline T,$
and so $xy\in \overline{K}\;\overline{ T}$.
Since $\overline K\subset G$ and $G^*$ is invariant, this contradicts the fact that
$\overline{T}\cap (G^*G^*)=\varnothing$. Therefore, $KT$ must be right translation-compact.
\end{proof}

In general neither  assertion (i) nor  assertion (ii)    of the preceding theorem implies  assertion (iii), as the following simple example shows.

\begin{example} \label{example} \begin{enumerate}\item In Theorem \ref{sprime},
Statement (ii) does not imply (iii) in general even if $C_0(G)\subseteq \A(G).$
In \cite{C1} Chou constructed
 (see also \cite[Section 5]{FG1}) a translation-finite set (even a t-set) $T$ in a discrete group $G$  which is not a $\B(G)$-set. Since, by Theorem \ref{union}, below,   $FT$ is translation-finite for every finite subset $F$ of $G$, we see that (ii) does not imply (iii).
\item
Also, Statement (i) may not imply (ii) when $C_0(G)\not \subseteq \A(G)$. Consider to that effect the discrete group $G$ obtained from \emph{forgetting} the topology of a compact connected semisimple Lie group. By van der Waerden's continuity theorem, see \cite[Corollary 5.65]{hofmorr06},  $\epsilon_{\ap}\colon G\to G^\ap$ is then  a group isomorphism. Groups with that property are known as   \emph{van der Waerden}, or sometimes \emph{self-Bohrifying}, groups (see \cite{comfrobe87,hartkunen02,shtern} for more examples and information on such  groups).  If $G$ is a van der Waerden group,  then $G^\ast=\emptyset$ and, hence,  Statement (i)  holds   trivially for any $T\subset G$, while Statement (ii) obviously fails for some $T$'s.\end{enumerate}
\end{example}

But, if $T$ is an $\A(G)$-set then the statements of Theorem \ref{sprime} are all equivalent, since (ii) clearly implies (iii). In particular,
we have the following corollary.

\begin{corollary}\label{cor:luc} Let $G$ be a locally compact group, $\A(G)$ be an admissible subalgebra of $\luc(G)$ and  $T$ be a right  uniformly discrete subset of $G.$ If  $\luc_\ast(G)\subseteq \A(G), $ then all the statements in Theorem \ref{sprime} are equivalent.
\end{corollary}

\begin{proof}
We only have to show that (ii)$\implies$(iii). Suppose that $T$ is right  uniformly discrete with respect to some relatively compact neighbourhood $U$ of $e$. Then statement (ii) implies  that $UT$ is  right translation-compact, and so by Theorem \ref{lwap}, $T$ is an $\luc_\ast(G)$-set. Thus, $T$  is an $\A(G)$-set.

If $\A(G)=\luc(G)$, then Theorem \ref{lwap} is not necessary, as $T$, being right uniformly discrete, is already an $\luc(G)$-set.
\end{proof}

\begin{remark} \label{mirror}
The \emph{left} version of Theorem \ref{sprime} can be proved with only changing the product $G^* G^*$ by $G^*\squ G^*$ and right translation-compact  by left translation-compact, the left version of  Lemma \ref{general**}  has then to be used. It follows, in particular, as in Corollary \ref{cor:luc}, that if $T$ is  left uniformly discrete and  $\ruc_\ast(G)\subseteq \A(G)$, then all the statements in the left analogue theorem are equivalent.

When $\A(G)= \wap(G),$ the two products coincide, and so we have the following corollary.
\end{remark}

 \begin{corollary} \label{cor:wap} Let $G$ be a locally compact $E$-group, and $T$ be a right uniformly discrete $E$-subset of $G$. The  following  statements are then  equivalent:
\begin{enumerate}[label=\textup{(\roman{enumi})}]

\item\label{(ii)} $\overline T\cap G^*G^*=\emptyset.$
\item[\textup{(ii)}$^\prime$] \label{(iii)p} $KT$ is  translation-compact for every relatively compact subset $K$ of $G.$
\item[\textup{(iii)}$
^\prime$] \label{(iii)pp} $T$ is a  $\wap(G)$-set.
\item[\textup{(iv)}$^\prime$] \label{(iiii)} For some open relatively compact subset $V$ of $G$,    $V\overline{T_1}$  is open in $G^\wap $ for every $T_1\subseteq T$.
\item[\textup{(v)}] $\overline T\cap \overline{G^*G^*}=\emptyset.$
\end{enumerate}
\end{corollary}

\begin{proof} Note that by Theorem \ref{wap=wap_0}, statements (iii)$^\prime$ and (iii) of the previous theorem are the same. Note also that statement (iv)$^\prime$ means actually that the set $V\overline{T_1}$  is open in $G^\wap $ for any open set $V$ in $G$ since translations in $G^\wap$ by  elements of $G$ are homeomorphisms (as argued already in the proof of Theorem \ref{openlocals}).
Therefore, by Theorem \ref{openlocals}, $T$ is $\wap(G)$-set, and so again by Theorem \ref{wap=wap_0},  statements (iv)$^\prime$  and (iv) of Theorem \ref{sprime} are the same.
As a result only statement (ii)$^\prime$ seems to differ  from (ii) of Theorem \ref{sprime},
and accordingly we have only to check that (i)$\implies$(ii)$^\prime$ as (ii)$^\prime\implies$(iii)$^\prime$ is clear.  But since both products ($xy$ and $x\squ y$) coincide on $G^\wap$,  it follows from Theorem \ref{sprime} and   Remark \ref{mirror} that  $T$ is both right and left translation-compact when (i) is assumed. Thus, it is also clear that  (i)$\implies$(ii)$^\prime$.

Finally, the implication (ii)$^\prime\implies$(iii)$^\prime$  follows either from  Lemma   \ref{lem:lucint}
or Theorem  \ref{wap=wap_0}.
\end{proof}

\begin{corollary} \label{cor:wapab} Let $G$ be a locally compact Abelian group, and let $\A(G)$ be an admissible algebra  with $\wap(G)\subseteq \A(G)\subseteq\luc(G)$. If $T$ is a right uniformly discrete subset of $G$,
the following  statements are equivalent:
\begin{enumerate}[label=\textup{(\roman{enumi})}]
\item  $\overline T\cap G^*G^*=\emptyset.$
\item[\textup{(ii)}$^\prime$]   $KT$ is  translation-compact for every relatively compact subset $K$ of $G.$
\item[\textup{(iii)}] For some  open relatively compact subset $V$ of $G$, $VT$ is  right translation-compact and $T$ is an  $\A(G)$-set.
\item[\textup{(iv)}]  For some  open relatively compact subset $V$ of $G$,
$VT$ is  translation-compact and
 $V\overline{T_1}$  is open in $G^\A $ for every $T_1\subseteq T.$
\item[\textup{(v)}] $\overline T\cap \overline{G^*G^*}=\emptyset.$
\end{enumerate}
\end{corollary}

\begin{proof}
When $G$ is Abelian,  left and right translation-compact sets are the same. Since $\wap(G)$-sets are $\A(G)$-sets, we can repeat the  proof of Corollary \ref{cor:wap}.
\end{proof}

\begin{remarks}\label{remstc}
\begin{enumerate}[label=\textup{(\roman{enumi})}]
 \item As a consequence of Corollary \ref{cor:luc}  we find that the sets constructed in Example \ref{t-set} satisfy all properties (i)--(v) of Theorem \ref{sprime}, for $\A(G)=\luc(G)$. If $G$ is an $SIN$-group, then Corollary \ref{cor:wap} implies that these sets  satisfy also the properties  for the $C^*$-algebra $\wap(G)$.

\item
 Corollary \ref{cor:wap} fails for $\A(G)=\B(G)$, as assertions (ii) and (iii) of Theorem \ref{sprime} are no longer equivalent in this case by Example \ref{example}(i).

 \item As a quaint consequence of  Theorem \ref{sprime},
 we see that $G^\ast G^\ast$ may be used in many cases as a criterion to test openness in $G^\A,$
since, under the  conditions stated in Theorem \ref{sprime}, $\overline T\cap G^*G^*=\emptyset$
implies that $U\overline T_1$ is open in $G^\A$ for any open subset $U$ of $G$ and every subset $T_1$ of $T$.
\end{enumerate}
\end{remarks}

Theorem \ref{sprime}  can be used as a tool to locate  strongly prime points in $G^\A$. We exploit this in the following corollaries.
  We start with a general result. When specialized to $G^\luc$ and $G^\wap$  we obtain a full characterization of strongly prime points.

   \begin{corollary} \label{sprime:subluc} Let $G$ be a locally compact group and $\A(G)$ be an admissible subalgebra of $\luc(G)$ containing $C_0(G)$. Then every strongly prime point in $G^\A$ is in the closure of some  right uniformly discrete subset $T$ of $G$
such that $KT$ is right translation-compact  for every relatively compact subset $K$ in $G$. In particular, strongly prime points are in the closure of right translation-finite uniformly discrete subsets of $G$.

If $\A(G)= \wap(G)$ or $G$ is Abelian, then the set $KT$ is translation-compact.
\end{corollary}

\begin{proof}
Suppose  $x\in G^\A$ is strongly prime and choose  a closed neighbourhood $C$  of $x$ in $G^{\A}$ such that
$C\cap (G^*G^*)=\emptyset.$ As argued in Corollary \ref{Veech}, we may take
a right uniformly discrete subset $X$ of $G$ such that $x\in \overline X,$
and let $T=C\cap X$. Then $x\in \overline T$ since $O\cap (C\cap X)=(O\cap C)\cap X\ne\emptyset$ whenever $O$
is a neighbourhood of $x$ in $G^{\A}.$
Since $\overline T\subseteq C$, $\overline T\cap G^*G^*=\emptyset$ and so by (i)$\Longrightarrow$(ii) of Theorem \ref{sprime}, we see that $KT$ is right translation-compact for every relatively compact subset $K$ of $G$.

In particular,  $T$ is right translation-compact  and Proposition \ref{prop} implies that $T$ is in fact right translation-finite.
Thus, the second statement follows.

If $\A(G)= \wap(G),$ then we apply Corollary \ref{cor:wap}.  When $G$ is Abelian, the claim is clear.
\end{proof}

This leads to the complete characterization of strongly prime points in the $\A$-compactifications of $G$
when $\luc_\ast(G)\subseteq \A(G)\subseteq \luc(G)$.

\begin{theorem}{\label{t51}}
 Let $G$ be a locally compact group and $\A(G)$ be an admissible subalgebra of $\luc(G)$ with $\luc_\ast(G)\subseteq \A(G).$ Then  a point $p\in G^\A$ is strongly prime if and only if $p\in \overline{T}$, where $T$ is  a subset of $G$ which is right uniformly discrete with respect to some relatively compact neighbourhood $V$ of $e$ such that $VT$ is right translation-compact.
\end{theorem}

\begin{proof} The necessity follows from the previous corollary, and  the converse follows from Corollary \ref{cor:luc} since   sets with the properties stated in the claim are $\luc_\ast(G)$-sets by Theorem \ref{lwap},
and so they are $\A(G)$-sets. Now apply Corollary \ref{cor:luc}.
\end{proof}

We next prove the analogue of Theorem \ref{t51}  giving, in particular, the characterization of the strongly prime points in  $G^\wap$ when $G$ is a locally compact $SIN$-group.

\begin{theorem}{\label{t52}}
 Let $G$ be a locally compact group and $\A(G)$ be an admissible subalgebra of $\luc(G)$. Assume that at least one of the following conditions hold:
 \begin{enumerate}
 \item $G$ is an $SIN$-group and $\A(G)=\wap(G),$
\item $G$ is Abelian and $\wap(G)\subseteq \A(G),$
\end{enumerate}
a point $p\in G^\A$ is strongly prime if and only if  $p\in \overline{T}$,
where $T$ is  a subset of $G$ which is right uniformly discrete with respect to some relatively compact neighbourhood $V$ of $e$ such that $VT$ is right translation-compact.
\end{theorem}

\begin{proof}  Note first  that the set $T$
in the proof of Corollary \ref{sprime:subluc} is translation-compact in the first case by Corollary  \ref{cor:wap}, and it is clearly translation compact in the second case.
With this in mind, the necessity follows precisely as in Theorem \ref{t51} from Corollary \ref{sprime:subluc}.

As for the converse, let $T$ be right uniformly discrete with respect to some neighbourhood $V$ of $e$ such that $VT$ is translation-compact,
and let $p\in \overline T.$
Then by Lemma  \ref{lem:lucint}, $T$ is  a $\wap(G)$-set, and so in both cases $T$ is an $\A(G)$-set.
Therefore, we may apply Corollary \ref{cor:wap} or Corollary  \ref{cor:wapab}  to deduce that $p$ is a strongly prime point in $G^\A$.
\end{proof}

When $\A(G)\subseteq\wap(G),$ we have the following partial characterization.

\begin{corollary}\label{cor:a(g)} Let $G$ be a locally compact group, $\A(G)$ be an admissible subalgebra of $\wap(G)$ and $T$ be a right uniformly discrete $\A(G)$-set. Then all points in the closure of $T$ are strongly prime.
\end{corollary}

 \begin{proof}
   By Lemma \ref{cor:wap2},  $UT$ is translation-compact whenever $T$ is an $\A(G)$-set which is right uniformly discrete with respect to some relatively compact neighbourhood $U$ of $e$. We only have to apply (iii) implies (v) of Theorem \ref{sprime}.
 \end{proof}

We can restate Theorems \ref{t51} and  \ref{t52} as follows.
Let  $\mathcal T$ be the family of all subsets of $G$ which are right $U$-uniformly discrete with respect to some neighbourhood $U$ of $e$ and for which $UT$ is right translation-compact.

\begin{corollary}{\label{t53}} Let $G$ be a noncompact  locally compact group and $\A(G)$ be an admissible subalgebra of $\luc(G).$
Then $G^\A\setminus\overline{G^\ast G^\ast}=\bigcup_{T\in \mathcal T}\overline T^\A $ in each of the following cases:
\begin{enumerate}
\item If $\luc_\ast(G)\subseteq \A(G).$
 \item If $G$ is an $SIN$-group and $\A(G)=\wap(G).$
\item If $G$ is Abelian and $\wap(G)\subseteq \A(G).$ \end{enumerate}
\end{corollary}

\begin{remark}  Let as usual $\mathcal N$ be a fixed  base of neighbourhoods at $e$ made of relatively compact sets. Budak and Pym observed in \cite[page 170]{BP} that the local structure theroem does not hold at every point in $G^\wap.$
They only supported this observation by saying that $K(G^\wap)\cap VT^*$ cannot be a neighbourhood of a point in the minimal ideal $K(G^\wap)$ for any $V\in \mathcal N$. This may now be clarified: if
$V\overline T$ is open  in $G^\wap$ for some $V\in \mathcal N,$ then Corollary \ref{cor:wap} shows that
$V\overline T\cap G^*G^*=\emptyset$, and so  $V\overline T\cap K(G^\wap)=\emptyset$.

 As with Veech's Theorem,  the local structure theorem holds on a dense subset of $G^\A\setminus \epsilon_\A(G)$, for many admissible subalgebras $\A(G)$ of $\luc(G)$.
This will be the content of our next corollary.
\end{remark}

\begin{corollary} \label{A-base} Let $G$ be a locally compact group and $\A(G)\subseteq \luc(G).$ Then in each of the following cases:
\begin{enumerate}
\item $\luc_\ast(G)\subseteq \A(G),$
 \item $G$ is an $SIN$-group and $\A(G)=\wap(G),$
\item $G$ is Abelian and $\wap(G)\subseteq \A(G),$ \end{enumerate}
there exists an open subset $D$ of $G^\A$ such that $G\subset D,$ $D\setminus G$ is dense in $G^*$ and each point in $D$ has a neighbourhood of the form $V\overline T$, which is homemorphic to $V\times \beta T$, where $V\in \mathcal N$ and
$T\in \mathcal T.$ In other words, \[\{V\overline T:V\in\mathcal N\;\text{and} \;T\in \mathcal T\}\] is a base for the topology in $D.$
\end{corollary}

\begin{proof} Let $D=\bigcup_{T\in \mathcal T}\overline T^\A .$  By Theorem \ref{t53}, $D=G^\A\setminus\overline{G^\ast G^\ast}
$. This set is clearly open, contains $G$, and has $D\setminus G$ dense in $G^*$
since each open set in $G^\A$ which has nonempty intersection with $G^*$ must contain a non-relatively compact subset $X$ of $G.$
By Example \ref{t-set}, $X$ contains a right uniformly discrete subset $T$ such that $VT$ is translation-compact for some $V\in\mathcal N.$
The rest follows from Corollaries \ref{cor:luc}, \ref{cor:wap} and \ref{cor:wapab}. \end{proof}

\subsection{On the union of translation-compact sets} \label{TCsets}

We end this section with the question of whether a finite union
of right translation-compact stays right translation-compact.
When $G$ is discrete and the sets are translation-finite, this was proved
in  \cite{rup}. More recently, this fact was proved in \cite[Lemma 5.1]{FP} again when $G$ is discrete but the sets are right translation-finite. The proofs in \cite{rup} and \cite{FP} are completely different. Where the former is a consequence of the characterization of translation-finite subsets as approximable $\wap$-interpolation sets, the latter is more direct and relies on combinatorial arguments.

In \cite[Theorem 4.16]{FG1}, we extended the main theorem on approximable $\wap$-interpolation sets proved in \cite{rup} to the class of
locally compact $E$-groups,  and deduced that if
$T_1$ and $T_2$ are subsets of an $E$-set in the group $G$  such that $T_1\cup T_2$ is right uniformly discrete with respect to some neighbourhood $U$ of $e$ and $UT_1$ and $UT_2$ are translation-compact, then
$VT_1\cup VT_2$ is right translation-compact for some neighbourhood $V$ of $e$, see \cite[Corollary 4.15]{FG1}.

All these results appear now as a direct consequence of Theorem \ref{sprime} or, rather, of its Corollary \ref{cor:luc}, for any locally compact group.

\begin{theorem} \label{union} Let $G$ be a locally compact group.
  The union of finitely many right translation-compact sets is right translation-compact.
  \end{theorem}
\begin{proof} It will suffice to prove that the union of two right translation-compact sets stays
  right  translation-compact.
  Let $T_1,\: T_2\subseteq G$ be right translation-compact and  let $T=T_1\cup T_2$. With  no loss of generality, we may assume that $T_1$ and $T_2$ are disjoint.
Fix a compact neighbourhood $U$ of $e$ and let $X$ be a maximal right $U$-uniformly discrete subset of $T$.  Then $T\subseteq U^2X$, and so $\overline T\subseteq U^2\overline X$ (the closure is taken in $G^\luc$).
  If $T$ is not right translation-compact, we can apply Corollary  \ref{cor:luc} and find $p\in \overline{T}\cap G^\ast G^\ast$. This implies that $\overline{X} \cap G^\ast G^\ast\neq    \emptyset$.

 Since $X=(X\cap T_1)\cup (X\cap T_2)$, we deduce that either $\overline{X\cap T_1} \cap G^\ast G^\ast\ne\emptyset$ or $\overline{X\cap T_2} \cap G^\ast G^\ast\ne\emptyset.$ This would imply,  again by  Corollary  \ref{cor:luc}, that either $(X\cap T_1)$ or $(X\cap T_1)$ is not right translation-compact, a contradiction with $T_1$ and $T_2$ being right translation-compact.
\end{proof}

\section{Sets with zero-mean} \label{0meansets}
Lemma \ref{general**} and Theorem \ref{sprime} have further  interesting  consequences related to the functions in $\A(G)$ which are annihilated by the left invariant means when such means exist on $\A(G)$. We end the paper  addressing these consquences.   Our first application is  Theorem \ref{suppinv}. This theorem  was obtained by Chou in  \cite[Lemma 2.5]{C4} when $G$ is a discrete group,
$T$ is translation-finite and $\A(G)=\wap(G)$. For this result, Chou used his characterization that
$T$ is translation-finite if and only if it is a $\wap(G)$-set. This theorem was also proved
in \cite[Theorem 11]{FP} when $G$ is an infinite,  left amenable, discrete group.
and $\A(G)=\ell_\infty(G)$. A similar result can also be found in \cite[Corollary 3.5.3]{GH} when $G$ is a discrete Abelian group, $\A(G)=\ap(G)$ and $T$ is an $I_0$-set, the proof however in this case is different and relies on purely  harmonic analytic tools.

The first part of our second theorem, Theorem \ref{suppinv2}, can be found in \cite[Proposition 9.13]{DLS} when $G$ is a left amenable discrete semigroup and $\A(G)=\ell_\infty(G)$. The second part of the theorem was proved
in \cite[Theorem 9.21]{DLS} when $G$ is an  infinite, left-amenable, cancellative
semigroup and $\A(G)=\ell_\infty(G)$.

By Theorem \ref{openlocals},  if $\A(G)$ is  a unital left tranlation invariant  C$^*$-subalgebra
of $\luc(G)$ and $T$ is any $\A(G)$-set in $G$, then  the set $ U\overline T$ (the closure in  $G^\A$) is  open in $G^\A$ for every open set $U$ in $G.$
So, regarding an element $\mu\in \A(G)^*$ as a Borel measure on $G^\A,$ we may well consider $\mu(U\overline T)$, and recall that  \begin{align*}\mu(U\overline T) = \sup\{ \mu(\phi): \phi\in \CB(G^\A), 0\le \phi \le 1, \, \supp  \phi\subseteq U\overline T\}.
 \end{align*}
 Now if $\phi \in \CB(G^\A)$ and $ \supp  \phi \subseteq U\overline T$, then $\restr{\phi}{G}\in \A(G)$ and
 $\supp \restr{\phi}{G}\subseteq U T.$
So  \begin{equation}\label{*}\mu(U\overline T)\le \sup\{ \mu(f):f\in \A(G), 0\le f\le 1,\supp f\subseteq U T\}. \end{equation}

Note that when $G$ is discrete and $T$ is an $\A(G)$-set, the characteristic function $\chi_T$ of $T$ is in $\A(G)$ and $ \mu(\overline T)$ is simply $\mu(\chi_T)$.

\begin{theorem} \label{suppinv}Let $G$ be a noncompact locally compact group, and $\A(G)$ be a unital, left tranlation invariant, amenable $C^*$-subalgebra
of $\luc(G)$
and $T$ be any $\A(G)$-set in $G$.
If $UT$ is right translation-compact for some neighbourhood $U$ of $e$,
then
$\mu(V\overline T)=0$ for every relatively compact neighbourhood $V$ of $e$ and every left invariant mean $\mu$ in $\A(G)^*.$ If, in addition $G$ is $\sigma$-compact, then $\mu(G\overline T)=0$ for every left invariant mean $\mu$ in $\A(G)^*.$
\end{theorem}

\begin{proof}  Suppose, otherwise, that $\mu(V\overline T)>0$ for some left invariant mean $\mu\in \A(G)^*$ and some relatively compact neighbourhood $V$ of $e,$ and suppose first that $V\subseteq U.$

We claim that  there exists $s\in G$, $s\notin V$ such that
 $\mu(sV\overline T\cap V\overline T)>0.$ Otherwise, let $n\in\N$ be such that $n\mu(V\overline T)>1$ and let $s_1, s_2,...,s_n$ be $n$ distinct elements  in $G$ not belonging to $V$.
Since $\mu(s_iV\overline T\cap s_jV\overline T)=0$ for every $1\le i<j\le n,$ we see that
 \[1\ge \mu(\bigcup_{k=1}^ns_kV\overline T)=\sum_{k=1}^n\mu(s_kV\overline T)=n\mu(V\overline T),\]  which is absurd.

So we may start with $s_1\in G,$ $s_1\notin V$ such that $\mu(s_1V\overline T\cap V\overline T)>0.$
Since $T$ is an $\A(G)$-set, we know  from Theorem \ref{openlocals} that $V\overline T$ is open in $G^\A,$
and so it can easily be verified that $s_1V\overline T\cap V\overline T\subseteq \overline{s_1VT\cap V T}$
(for if $p\in s_1V\overline T\cap V\overline T$ and $(s_1vt_\alpha)$ is a net in $s_1VT$ converging to $p$ in $G^\A,$
then this net is eventually in $V\overline T\cap G=VT$. Thus, the net $(s_1vt_\alpha)$ is eventually in $s_1VT\cap VT,$
and so $p\in\overline {s_1VT\cap VT}$.)
It follows that $\mu(\overline{ s_1VT\cap VT})>0.$
Put $A_0= VT$ and $A_1=s_1^{-1}VT\cap VT=s_1^{-1}A_0\cap A_0.$
Then clearly $A_1\cup s_1A_1\subseteq A_0$ and $\mu(\overline {A_1})>0.$
So we may pick again $s_2\in G,$ $s_2\notin s_1V\cup V$,   such that $\mu(s_2A_1\cap A_1)>0.$
Put then $A_2=s_2^{-1}A_1\cap A_1$, and note that in turn \[A_2\cup s_2A_2\subseteq A_1\quad\text{ and}\quad\mu(\overline {A_2})>0.\]
By induction, we obtain an  infinite set
$L=\{s_n:n\in\N\}$ which is right uniformly discrete and a decreasing family $\{A_n:n\in\N\}$ of subsets of $VT$ such that \[s_nA_n\cup A_n\subseteq A_{n-1}\quad\text{and}\quad \mu(\overline {A_n})>0\quad\text{ for every}\quad n\in\N.\]
Let now   $\{s_{k_1}, ..., s_{k_n}\}$ be any finite subset of $L$ with $k_1<...<k_n$. Then we have \[
A_{k_{n}}= \bigcap_{i=1}^nA_{k_i}\subseteq \bigcap_{i=1}^ns_{k_i}^{-1}A_{k_i-1}\subseteq \bigcap_{i=1}^ns_{k_i}^{-1}VT.\]
Since $\mu(\overline {A_{k_n}})>0,$ we see that $A_{k_n}$ is not relatively compact in $G$ (otherwise taking  sufficiently many pairwise disjoint left translates of $A_{k_n}$ will contradict the fact that $\mu$ is a mean). Therefore, $VT$ cannot be right translation-compact, and so $UT$ cannot be right translation-compact either.

If $V$ is an arbitrary relatively compact neighbourhood of $e$ and $\mu(V\overline T)>0,$
then $V$ may be covered by finitely many left translates of a fixed neighbourhood $W$ of $e$ with $W\subseteq V.$ Accordingly, $\mu(W\overline T)>0$, which is not possible as we have just seen.

The second part of the claim, when $G$ is $\sigma$-compact, is clear.
\end{proof}

\begin{remark}
The condition that $T$ is an $\A(G)$-set is necessary in the previous theorem. For, let $G$ be the group $SL(2,\R)$, and let   $T$
be a right uniformly discrete set with respect to some neighbourhood $U$ of $e$
such that $UT$ is right translation-compact (Example \ref{t-set} can be used to construct such a set). In this case $\wap(G)=\cnaught(G)\oplus\C1$, see \cite{C0}. Then
$G^\wap$ is the one-point compactification $G_p=G\cup\{p\}$ where $p$ is the point at infinity (the zero point in $G^\wap$),
and so $p$ is the invariant mean in $\wap(G)^*$. As before, we let $p$  denote also the corresponding Borel measure on $G^\wap$ (i.e., the point mass measure). By the joint continuity property, we see that $U\overline T$ is a Borel set in $G^\wap$.  Now since $T$ is not relatively compact, any open set in $G_p$ containing $U\overline T$ must contain $p.$ Thus, \[p(U\overline T)=\inf\{p(O):O\;\text{is  open in}\; G_p\;\text{ and}\; U\overline T\subseteq O\}=1.\] \qed
 \end{remark}

As already noted in Lemma \ref{im}, and used throughout the paper, the support of the invariant mean on $\wap(G)$ is the minimal ideal $K(G^\wap)$ and so it is contained in $G^*G^*$.
 For the proof of this fact, it is crucial  that the operation in $G^\wap$ is {\it separately}  continuous (see \cite{deleglick}), a condition  that is not satisfied by many compactifications in this paper.  Following a different approach, we deduce from the results of this paper, that  for a large family of these compactifications, left invariant means are supported  in $\overline{G^*G^*}$.

\begin{theorem} \label{suppinv2}  Let $G$ be a topological group and $\A(G)$ be an amenable admissible subalgebra of $\CB(G)$ and let $\mu$ be a left invariant mean on $\A(G).$
 Then the support of $\mu$ is a closed left ideal of $G^\A$.

Moreover, if $G$ is a noncompact, locally compact group and $\A(G)\subseteq \luc(G),$
then the support of $\mu$ is contained in $\overline{G^\ast G^\ast}$ in the following cases:
\begin{enumerate}
\item $\luc_\ast(G)\subseteq \A(G),$
\item $G$ is Abelian and $\wap(G)\subseteq\A(G)$.
\end{enumerate}
\end{theorem}

\begin{proof}
We first check that $\supp \mu$ is a closed left ideal of
$G^\A$ when $\mu$ is   a left invariant mean on $\A(G).$ Since $\supp\mu$ is closed,
 we only need to show that
$\epsilon_\A(G)\supp \mu \subseteq \supp \mu$.  Let $x\in \supp \mu$, $s\in G$ and $f$
be a nonnegative function in $\A(G)$ with $f^\A(\epsilon_\A(s)x)\ne 0.$
Then $({}_sf)^\A$ is a nonnegative function on $G^\A$ with $({}_sf)^\A(x)\ne0,$
and so $\mu({}_sf)\ne0.$ But $\mu({}_sf)=\mu(f).$ Hence, $\mu(f)\ne0$ showing that
$\epsilon_\A(s) x\in \supp \mu,$ as required for the first part of our claim.

Suppose now that $G$ is a noncompact, locally compact group.
To see that $\supp\mu$ is contained in $\overline{G^\ast G^\ast},$ let $x\notin\overline{G^\ast G^\ast}.$ Then, $x$ is a strongly prime point in $G^\A$.
We now apply Theorem \ref{t51}, for Statement (i), and Theorem \ref{t52} for Statement (ii). We obtain that  $x$ is in the closure of some subset $T$ of $G$ which is right uniformly discrete with respect to some relatively compact neighbourhood $U$ of $e$ such that $UT$ is right translation-compact (if Statement (i) is assumed), or translation-compact (under  Statement (ii)).

Now Theorem \ref{lwap} for statement (i), and Lemma \ref{lem:lucint} for Statement (ii), show that $T$ is an $\A(G)$-set.   Theorem \ref{suppinv} implies then  that
 $\mu(U\overline T)=0$ (the closure is taken in $G^\A$). Since by Theorem \ref{openlocals},  $U\overline T$ is in each case an open neighbourhood of $x$ in $G^\A,$ we have that $x\notin \supp(\mu),$
as required.
\end{proof}

We return to what we promised in Remark \ref{promisse}. The following corollary is straightforward now that we have Theorem \ref{suppinv2}.

\begin{corollary} Let $G$ be a locally compact group. Then
$\luc_a(G)\subseteq \luc_0(G).$
\end{corollary}
\begin{remark}
Most of the  concepts discussed  in this paper  make also sense  in the semigroup case.
It would be interesting to know which of the results  can be proved in this  setting.
\end{remark}
\noindent {\sc Acknowledgements.}
We would like to thank the referee for his/her careful reading of the paper and for  the useful comments.

This paper started  with the visits of  the
first author to  University of Jaume I in Castell\'on in
December 2013-January 2014 and of  the second author to  Oulu in 2015.
 We would like to thank both departments of mathematics
in Castell\'on and Oulu for letting each of us use their facilities.
The work was partially supported by the program {\it Short-Term International Research Visits, University of Oulu}.
This support is gratefully acknowledged.

The second-listed author has also been supported by project $\mbox{MTM2016-77143-P}$ (AEI/FEDER, UE) from Ministerio de Econom\'{\i}a, Industria y Competitividad, Spain.
\enlargethispage{16mm}

\end{document}